\documentclass[12pt]{article}
\usepackage{amssymb}
\usepackage{amsmath}
\usepackage{mathrsfs}
\usepackage{amsthm}
\usepackage{graphicx}
\usepackage{subfigure}

\theoremstyle{definition} \newtheorem{definition}{Definition}[section]
\theoremstyle{definition} \newtheorem{theorem}{Theorem}[section]
\theoremstyle{definition} 
\theoremstyle{definition} 
\theoremstyle{definition} \newtheorem{Proposition}{Proposition}[section]

\newtheorem{remark}{Remark}[section]

\begin{document}

\title{Symplectic methods based on Pad$\acute{e}$ approximation for some stochastic Hamiltonian systems}

\author{Liying Sun$^{a}$ and Lijin Wang$^{a}$\noindent {\footnote{Corresponding author, E-mail:ljwang@ucas.ac.cn}}\\
{\small a. School of Mathematical Sciences, University of Chinese Academy of Sciences, Beijing \textup{100049}}\\
}

\maketitle

\begin{abstract}
In this article, we introduce a kind of numerical schemes, based on Pad$\acute{e}$ approximation, for two  stochastic Hamiltonian systems which are treated separately. For the linear stochastic Hamiltonian systems, it is shown that the applied Pad$\acute e$ approximations $P_{(k,k)}$ give numerical solutions that inherit the symplecticity and the proposed numerical schemes based on $P_{(r,s)}$ are of mean-square order $\frac{r+s}{2}$ under appropriate conditions. In case of the special stochastic Hamiltonian systems with additive noises, the numerical method using two kinds of Pad$\acute e$ approximation $P_{(\hat r,\hat s)}$ and $P_{(\check r,\check s)}$ has mean-square order $\check r+\check s+1$ when $\hat r+\hat s=\check r+\check s+2$. Moreover, the numerical solution is symplectic if $\hat r=\hat s$.
 \\
{\it Keywords:} Pad$\acute{e}$ approximation; linear stochastic Hamiltonian systems; symplecticity; mean-square order;
\end{abstract}

\section{Introduction}
\label{1}
It is well known that the flow $\varphi(t)$ of the deterministic Hamiltonian system
\begin{equation}\label{1.1}
\begin{split}
&dp=-\frac{\partial H(p,q)}{\partial q}dt,\quad p(0)=p_0,\\
&dq=\frac{\partial H(p,q)}{\partial p}dt,\quad q(0)=q_0,
\end{split}
\end{equation}
is symplectic, i.e. $$dp(t)\wedge dq(t)=dp_0\wedge dq_0, \forall t\geq 0,$$ which also can be depicted as follows
\begin{equation*}\label{1.1.5}
\left(\frac{\partial \varphi(t)}{\partial y_0}\right)^{\top}J\left(\frac{\partial \varphi(t)}{\partial y_0}\right)=J,\quad y_0=(p_0,q_0)^{\top}.
\end{equation*}
In recent years, many numerical analysts have been devoting to constructing symplectic numerical methods for such systems, that is, numerical approximations $(p_{n},q_{n})$ that preserve the symplecticity of the underlying Hamiltonian systems, characterised by $$dp_{n+1}\wedge dq_{n+1}=dp_{n}\wedge dq_{n},\quad \forall n\in Z, n\geq 0.$$ For example, the numerical methods based on the P$\acute{a}$de approximation are effective symplectic solvers of linear Hamiltonian systems (LHS) \cite{feng}, i.e. the Hamiltonian system (\ref{1.1}) where Hamiltonian function $H(p,q)$ is of the quadratic form
$$H(p,q)=\frac{1}{2}(p^{\top},q^{\top})C
\begin{pmatrix}
p \\ q
\end{pmatrix}, \quad C^{\top}=C.$$ For the LHS,  
symplectic numerical schemes based on the Pad$\acute e$ approximation $P_{(k,k)}(C)$ take the form
\begin{equation}
\label{1.2}
\begin{bmatrix}
p_{n+1} \\ q_{n+1}
\end{bmatrix}
=P_{(k,k)}(C)
\begin{bmatrix}
p_{n} \\  q_{n}
\end{bmatrix}
=\left[D_{(k,k)}(C)\right]^{-1}N_{(k,k)}(C)
\begin{bmatrix}
p_{n} \\ q_{n}
\end{bmatrix},
\end{equation}
with the matrix polynomials
\begin{equation}
\label{1.3}
\begin{split}
&N_{(k,k)}(C)=\sum_{i=0}^{k} \frac{(2k-i)!k!}{(2k)!i!(k-i)!}(hJ^{-1}C)^{i},\\
&D_{(k,k)}(C)=\sum_{i=0}^{k} \frac{(2k-i)!k!}{(2k)!i!(k-i)!}(-hJ^{-1}C)^{i},
\end{split}
\end{equation}
and a standard antisymmetric matrix $J=\begin{bmatrix}
 0 &I_{n}\\ -I_{n} &0\end{bmatrix}$ for $k=1,2,\cdots$. It is proved that the difference scheme (\ref{1.2}) is not only symplectic, but also of $2k$-$th$ order of accuracy (\cite{feng}).

Stochastic Hamiltonian systems (SHSs) constitute a rather important class of stochastic systems having the  property of symplecticity (\cite{milbook,mil1,mil2}). The problem of constructing special numerical methods preserving the symplectic structure of SHSs is of great current interest (\cite{kloeden,milbook,mil1,mil2}). We consider the linear stochastic Hamiltonian systems (\ref{1.4}) of dimension $d = 2n$ in the sense of Stratonovich
\begin{equation}\label{1.4}
\begin{split}
&dP=-\frac{\partial H_0(P,Q)}{\partial Q}dt-\sum_{i=1}^m \frac{\partial H_i(P,Q)}{\partial Q}\circ dW^{i}(t), \quad P(t_0)=p,\\
&dQ=\frac{\partial H_0(P,Q)}{\partial P}dt+\sum_{i=1}^m \frac{\partial H_i(P,Q)}{\partial P}\circ dW^{i}(t), \quad\quad Q(t_0)=q,
\end{split}
\end{equation}
and stochastic Hamiltonian systems with additive noises in the sense of It$\hat{o}$ (\ref{1.4.5})
\begin{equation}\label{1.4.5}
\begin{split}
&d\tilde{P}=-\frac{\partial \tilde{H}_0(\tilde{P},\tilde{Q})}{\partial \tilde{Q}}dt-\sum_{i=1}^m \frac{\partial \tilde{H}_i(\tilde{P},\tilde{Q})}{\partial \tilde{Q}}dW^{i}(t), \quad \tilde{P}(t_0)=\tilde{p},\\
&d\tilde{Q}=\frac{\partial \tilde{H}_0(\tilde{P},\tilde{Q})}{\partial \tilde{P}}dt+\sum_{i=1}^m \frac{\partial \tilde{H}_i(\tilde{P},\tilde{Q})}{\partial \tilde{P}}dW^{i}(t), \quad\quad \tilde{Q}(t_0)=\tilde{q},
\end{split}
\end{equation}
where $P$, $Q$, $\tilde{P}$, $\tilde{Q}$, $p$, $q$, $\tilde{p}$ and $\tilde{q}$ are n-dimensional column-vectors with the components $P_{l}$, $Q_{l}$, $\tilde{P}_{l}$, $\tilde{Q}_{l}$, $p$, $q$, $\tilde{p}_{l}$, $\tilde{q}_{l}$, $l=1,\cdots,n$, and $(W^{1}(t),\cdots,W^{m}(t)$ is $m$-n-dimensional standard Wiener
process. In the stochastic differential equations (\ref{1.4}) and (\ref{1.4.5}), $H_{i}(P,Q)$, $i=0,\cdots,m$, $\tilde{H}_0(\tilde{P},\tilde{Q})$ are of quadratic forms, i.e.
\begin{equation*}
\label{1.5}
H_{i}(P,Q)=\frac{1}{2}(P^{\top},Q^{\top})C^{i}
\begin{pmatrix}
P \\ Q
\end{pmatrix},\quad \tilde{H}_0(\tilde{P},\tilde{Q})=\frac{1}{2}(\tilde{P}^{\top},\tilde{Q}^{\top})\tilde{C}^0
\begin{pmatrix}
\tilde{P} \\ \tilde{Q}
\end{pmatrix},
\end{equation*}
where $C^{i}$, $i=0,\cdots,m$ and $\tilde{C}^0$ are $2n\times 2n$ symmetric matrices. $$\tilde{H}_{i}=<\tilde{C}^{i}_{1},  \tilde{P}>-<\tilde{C}^{i}_{2}, \tilde{Q}>,\quad i=1,\cdots,m$$ where $\tilde{C}^{i}_{1}$, $\tilde{C}^{i}_{2}$, $i=1,\cdots,m$ are $n$-dimensional constant column-vectors . In addition, we suppose that the coefficients
of the LSHS  are sufficiently smooth functions defined for $(t; p,q)\in [t_0, t_0 + T]\times R^{2n}$ which guarantee the existence and the uniqueness of the solution in the interval $[t_{0},t_{0}+T]$ (see \cite{mao,oksen}).

Let $X(t;t_0,p,q)$ 
= $(P(t;t_0,p,q)^{\top},Q(t;t_0,p,q)^{\top})^{\top}$, $t_0\leq t\leq t_0 + T$ and $Z(t;t_0,\tilde{p},\tilde{q})$= $(\tilde{P}(t;t_0,p,q)^{\top}$,$\tilde{Q}(t;t_0,\tilde{p},\tilde{q})^{\top})^{\top}$, $t_0\leq t\leq t_0 + T$ be the solutions of the linear stochastic Hamiltonian system (\ref{1.4}) and the linear stochastic Hamiltonian system with additive noises (\ref{1.4.5}) respectively. A more detailed notation is $X(t; t_0, p, q, \omega)$ and $Z(t; t_0, \tilde{p}, \tilde{q}, \omega)$, where $\omega$ is an elementary event. It is known that $X(t; t_0,p,q)$ and $Z(t; t_0,\tilde{p},\tilde{q})$ are phase flows (diffeomorphism) for almost every $\omega$, properties of which can be seen in e.g., \cite{Ikeda}. Moreover, we denote by $X_{k}=(P_{k}^{\top},Q_{k}^{\top})^{\top}$, $k=0,\cdots,N$ the numerical method for (\ref{1.4}) and by $Z_{k}=(\tilde{P}_{k}^{\top},\tilde{Q}_{k}^{\top})^{\top}$, $k=0,\cdots,N$ the numerical method for (\ref{1.4.5}), respectively, and $h=t_{k+1}-t_{k}$, $t_{N}=t_0+T$.

As the class of the symplectic difference schemes (\ref{1.2}) for the deterministic Hamiltonian systems (\ref{1.1}) have been proposed, it is meaningful to investigate, whether we could extend this approach to stochastic context and construct the symplectic numerical methods for stochastic Hamiltonian systems (\ref{1.4}) and (\ref{1.4.5}) by using the Pad$\acute e$ approximation. An outline of this paper is as follows: Section 2 is devoted to numerical methods $X^{(r,s)}(t+h;t,p,q)$, $m=1,2,3,4$ via using the Pad$\acute e$ approximation for linear stochastic Hamiltonian systems and a review of well known facts concerning some properties of the infinitesimal symplectic matrices. In Section 3, some numerical schemes $Z(t+h;t,\tilde{p},\tilde{q})$ based on the Pad$\acute e$. approximation  $P_{(\hat{r},\hat{s})}$ and $P_{(\check{r},\check{s})}$ are proposed for the special stochastic Hamiltonian systems with additive noises. In Section 4, we investigate the mean-square convergence of the proposed numerical approximations and prove that some of them preserve symplectic structure under appropriate conditions. Section 5 gives numerical tests. At last, Section 6 is a brief conclusion

\section{Symplectic numerical methods for linear Stochastic Hamiltonian systems}
\subsection{Linear Stochastic Hamiltonian systems(LSHS)}
\label{2}
Using $\nabla H_{i}(X)=C^iX=JA^{i}X$, $i=0,1,\cdots,m$, the canonical system (\ref{1.4}) becomes
\begin{equation}\label{2.1}
dX(t)=A^0Xdt+\sum_{i=1}^m A^{i}X\circ dW^{i}(t),\quad X(t_0)=(p,q)^{\top},
\end{equation}
the unique solution which is as follows:
\begin{equation}\label{2.2}
X(t)=\exp \left[(t-t_0)A^0+\sum_{i=1}^m (W^{i}(t)-W^{i}(t_0))A^{i}\right]X(t_0),
\end{equation}
with $ t_0\leq t\leq t_0 + T$. 
Denoting
\begin{gather*}
A^{i}=
\begin{bmatrix}
A^{i}_1& A^{i}_2 \\ A^{i}_3 & A^{i}_4,
\end{bmatrix}, i=0,1,\cdots,m,
\end{gather*}
where $A^{i}_{j}$, $i=0,\cdots, m$, $j=1,2,3,4$ are $n\times n$ constant matrices, and substituting the symmetric matrices $A^{i}$ into the right side of  (\ref{2.1}), we obtain
\begin{equation}\label{2.3}
\begin{split}
&dP=(A_1^0P+A_2^0Q)dt+\sum_{i=1}^m (A_{1}^{i}P+A_2^{i}Q)\circ dW^{i}(t),\quad P(t_0)=p,\\
&dQ=(A_3^0P+A_4^0Q)dt+\sum_{i=1}^m (A_{3}^{i}P+A_4^{i}Q)\circ dW^{i}(t),\quad Q(t_0)=q.
\end{split}
\end{equation}
Since $JA^{i}$, $i=0,1,\cdots,m$ are symmetric matrices, it is easy to check that $A^{i}_{1}=-A^{i\top}_{4}$, $A^{i}_{2}=A^{i\top}_{2}$, $A^{i}_{3}=A^{i\top}_{3}$, $i=0,\cdots, m$.

According to the fundamental theorem of Hamiltonian systems, the solution of a Hamiltonian system is a one-parameter symplectic group $G_{t}$ whose elements are called symplectic matrices, denoted by $Sp(2n)$. Therefore, the symplectic geometry serves as the mathematical foundation of Hamiltonian mechanics. Before stating the numerical methods for LSHS, let us introduce
some properties of infinitesimal symplectic matrices , which will be used in the proof of our main results Theorem \ref{th4} and Theorem \ref{th6}.
\begin{definition}
(\mbox{\cite{feng}}) A matrix B of order 2n is called symplectic, i.e. $B\in Sp(2n)$, if
\begin{equation*}
B^{\top}JB= J.
\end{equation*}
where $B^{\top}$ is the transpose of $B$. All symplectic matrices form a symplectic group $Sp(2n)$.
\end{definition}
\begin{definition}
({\mbox{\cite{feng}}}) A matrix B of order 2n is called infinitesimal symplectic, if
\begin{equation*}
JB + B^{\top}J = O.
\end{equation*}
All infinitesimal symplectic matrices form a Lie algebra with commutation operation $[A,B]= AB-BA$, denoted as $sp(2n)$. $sp(2n)$ is the Lie algebra of the Lie group $Sp(2n)$.
\end{definition}

\begin{remark}
The matrices $A^{i}$, $i=0,1,\cdots,m$, in (\ref{2.1}) are infinitesimal symplectic matrices.
\end{remark}
\begin{theorem}\label{th0.51}
(\cite{feng}) If $f(x)$ is an even polynomial, and $B\in sp(2n)$, then
\begin{equation*}
f(B^{\top})J =Jf(B).
\end{equation*}
\end{theorem}

\begin{theorem}\label{th0.52}
(\cite{feng}) If $g(x)$ is an odd polynomial, and $B\in sp(2n)$, then $g(B)\in sp(2n)$, i.e.,
\begin{equation*}
g(B^{\top})J + Jg(B) = O.
\end{equation*}
\end{theorem}

\begin{theorem}\label{th0.53}
(\cite{feng}) Matrices
$S = M^{-1}N\in Sp(2n)$, iff
\begin{equation}\label{2.4}
MJM^{\top}= NJN^{\top}.
\end{equation}
\end{theorem}

\subsection{Constructing numerical methods based on Pad$\acute{e}$ approximation}

As is well known, the exponential function $\exp{(M)}$ for $n\times n$ dimensional matrix $M$ has the Taylor's expansion
\begin{equation}\label{3.6}
\exp{(M)}=I+\sum_{i=1}^{+\infty} \frac{M^{i}}{i!},
\end{equation}
and can be approximated by Pad$\acute{e}$ approximation as follows
\begin{equation}\label{3.7}
\exp{(M)} \sim P_{(r,s)}(M)= D_{(r,s)}^{-1}(M)N_{(r,s)}(M),
\end{equation}
with $$N_{(r,s)}(M)=I+\sum_{i=1}^{r} \frac{(r+s-i)!r!}{(r+s)!i!(r-i)!}M^{i}=I+\sum_{i=1}^{r}a_iM^{i},$$, $$D_{(r,s)}(M)=I+\sum_{i=1}^{s} \frac{(r+s-i)!s!}{(r+s)!i!(s-i)!}(-M)^{i}=I+\sum_{i=1}^{s}b_i(-M)^{i}.$$ In addition, (\ref{3.7}) makes the following equation holds (\cite{feng})
\begin{equation}\label{3.8}
\exp{(M)}-P_{(r,s)}(M)=c_{r+s+1}M^{r+s+1}+\sum_{i=r+s+2}^{+\infty}c_{i}M^{i}=O(M^{r+s+1}).
\end{equation}
where $c_{i}$, $i\geq{r+s+1}$ are constants.

Based on the Pad$\acute e$ approximation $P_{(r,s)}[(t_{n+1}-t_{n})A^0+\sum_{i=1}^{m}(W^{i}(t_{n+1})-W^{i}(t_{n}))A^{i}]$, we construct the following numerical methods for linear stochastic Hamiltonian systems (\ref{2.1})
\begin{equation}
\label{3.9}
\hat{X}_{n+1}^{(r,s)}=\left[I+\sum_{j=1}^{s} b_j(-B)^{j}\right]^{-1}\left[I+\sum_{j=1}^{r} a_j(B)^{j}\right]\hat{X}_{n}^{(r,s)},
\end{equation}
where $B=hA^0+\sum_{i=1}^{m}\Delta W^{i}_{n}A^{i}$
and $\Delta W^{i}_{n}=W^{i}(t_{n+1})-W^{i}(t_{n})$ are the increments of Wiener processes which can be substituted by $\sqrt{h}\xi^{i}$ and $\xi^{i}\sim N(0,1)$, $i=1,\cdots,m$ are independent random variables. 

Similar to the discussion in (\cite{mil2}), we truncate $\xi^{i}$ to another random variable $\zeta_{h}^{i}$ which is bounded. In detail,
\begin{equation*}\label{3.5}
\zeta_{h}^{i}=\left\{\begin{array}{l}\ \xi^{i},\,\,\,\,\mbox{if}\,\,\,\,|\xi^{i}| \leq A_{h},\\ \ A_{h},\,\,\,\,\mbox{if}\,\,\,\, \xi^{i} > A_{h},\\ \ -A_{h},\,\,\,\,\mbox{if}\,\,\,\, \xi^{i} < -A_{h},\end{array}\right.
\end{equation*}
where $A_{h}=\sqrt{2\ell|\ln h|}$, $\ell \geq 1$.
For the sake of simplicity, we denote the matrix polynomials $hA^0+{\mathbf{\sum}}_{i=1}^{m}\sqrt h \zeta_{h}^{i}A^{i}$ by $\bar B$, and the following numerical scheme:
\begin{equation}
\label{3.11}
X_{n+1}^{(r,s)}={\left[I+\sum_{j=1}^{s} b_j(-\bar B)^j\right]}^{-1}\left[I+\sum_{j=1}^{r} a_j\bar B ^j\right]X_{n}^{(r,s)}.
\end{equation}
with $A_{h}^{i}=\sqrt{2\ell|\ln h|}$, $\ell \geq 1$. It is interesting to observe that if $r=s=1$, we reattain the Euler centered scheme
\begin{equation}
\label{3.12}
X_{n+1}^{(1,1)}=X_{n}^{(1,1)}+\frac{1}{2}\bar B(X_{n}^{(1,1)}+X_{n+1}^{(1,1)}).
\end{equation}
If both $r$ and $s$ take the value 2, we obtain the scheme
\begin{equation}
\label{3.13}
X_{n+1}^{(2,2)}=X_{n}^{(2,2)}+\frac{1}{2}\bar B(X_{n}^{(2,2)}+X_{n+1}^{(2,2)})+\frac{1}{12}\bar B^2(X_{n}^{(2,2)}-X_{n+1}^{(2,2)}).
\end{equation}
Similarly, if $r=s$ equal 3 and 4, the following two kinds of numerical approximations arise, respectively,
\begin{equation}
\label{3.14}
X_{n+1}^{(3,3)}=X_{n}^{(3,3)}+(\frac{1}{2}\bar B+\frac{1}{120}\bar B^3)(X_{n}^{(3,3)}+X_{n+1}^{(3,3)})+\frac{1}{10}\bar B^2(X_{n}^{(3,3)}-X_{n+1}^{(3,3)}).
\end{equation}
\begin{equation}
\label{3.15}
X_{n+1}^{(4,4)}=X_{n}^{(4,4)}+(\frac{1}{2}\bar B+\frac{1}{84}\bar B^3)(X_{n}^{(4,4)}+X_{n+1}^{(4,4)})+(\frac{1}{24}\bar B^2+\frac{1}{1680}\bar B^4)(X_{n}^{(4,4)}-X_{n+1}^{(4,4)}).
\end{equation}
In Section \ref{4}, the properties of the methods $(\ref{3.12})\sim (\ref{3.15})$ will be analyzed.
\section{Symplectic numerical methods for stochastic Hamiltonian systems with additive noises}
\label{3}

We now turn our attention to stochastic Hamiltonian systems with additive noises (\ref{1.4.5}) mentioned in Section \ref{1},
\begin{equation*}
\begin{split}
&d\tilde{P}=-\frac{\partial \tilde{H}_0(\tilde{P},\tilde{Q})}{\partial \tilde{Q}}dt-\sum_{i=1}^m \frac{\partial \tilde{H}_i(\tilde{P},\tilde{Q})}{\partial \tilde{Q}}dW^{i}(t), \quad \tilde{P}(t_0)=\tilde{p},\\
&d\tilde{Q}=\frac{\partial \tilde{H}_0(\tilde{P},\tilde{Q})}{\partial \tilde{P}}dt+\sum_{i=1}^m \frac{\partial \tilde{H}_i(\tilde{P},\tilde{Q})}{\partial \tilde{P}}dW^{i}(t), \quad\quad \tilde{Q}(t_0)=\tilde{q},
\end{split}
\end{equation*}
where
$$\tilde{H}_0(\tilde{P},\tilde{Q})=\frac{1}{2}(\tilde{P}^{\top},\tilde{Q}^{\top})\tilde{C}^0
\begin{pmatrix}
\tilde{P} \\ \tilde{Q}
\end{pmatrix},\quad
\tilde{C}^0=\tilde{C}^{0\top},$$ and  $$\tilde{H}_{i}=<\tilde{C}^{i}_{1},  \tilde{P}>-<\tilde{C}^{i}_{2}, \tilde{Q}>,\quad i=1,\cdots,m.$$
By $\nabla \tilde{H}_0(Z)=\tilde{C}^{0}Z$ and $\nabla \tilde{H}_{i}(Z)=R_{i}={(\tilde{C}_{1}^{i\top},-\tilde{C}_{2}^{i\top})}^{\top}$, the stochastic Hamiltonian system with additive noises (\ref{1.4.5}) becomes
\begin{equation}\label{6.1}
dZ(t)=J^{-1}\tilde{C}^{0}Zdt+\sum_{i=1}^m J^{-1}R_{i}dW^{i}(t),\quad Z(t_0)=(\tilde{p},\tilde{q})^{\top},
\end{equation}
the exact solution of which is as follows:
\begin{equation}\label{6.2}
Z(t)=e
^{(t-t_0)J^{-1}\tilde{C}^{0}}Z(t_0)+\sum_{i=1}^m \int_{t_0}^{t}e^{(t-\theta)J^{-1}\tilde{C}^{0}}J^{-1}R_{i}dW^{i}(\theta),
\end{equation}
with $ t_0\leq t\leq t_0 + T$. 

An approach for the construction of the numerical schemes for stochastic differential equations (\ref{6.1}) is to replace the matrix exponential by the Pad$\acute e$ approximation. Approximating the matrix exponential $e
^{(t-t_0)J^{-1}\tilde{C}^{0}}$ and $e^{(t-s)J^{-1}\tilde{C}^{0}}$ by the Pad$\acute e$ approximations $P_{(\hat{r},\hat{s})}[(t_{n+1}-t_{n})J^{-1}\tilde{C}^{0}]$ and $P_{(\check{r},\check{s})}[(t_{n+1}-s)J^{-1}\tilde{C}^{0}]$ respectively, we obtain the numerical methods as follows
\begin{equation}\label{6.3}
\begin{split}
Z_{n+1}=&\left[I+\sum_{j=1}^{\hat{s}} b_j(-B_1)^{j}\right]^{-1}\left[I+\sum_{j=1}^{\hat{r}} a_j(B_1)^{j}\right]Z_{n}\\
&+\sum_{i=1}^m \int_{t_{n}}^{t_{n+1}}\left[I+\sum_{j=1}^{\check{s}} b_j(-B_2)^{j}\right]^{-1}\left[I+\sum_{j=1}^{\check{r}} a_j(B_2)^{j}\right]J^{-1}R_{i}dW^{i}(\theta),
\end{split}
\end{equation}
where $B_1=hJ^{-1}\tilde{C}^{0}$ and $B_2=(t_{n+1}-\theta)J^{-1}\tilde{C}^{0}$ with $\hat{r}+\hat{s}=\check r+\check s+2$ and $\check r, \check s\geq 1$. In particular, if $\hat{r}=\hat{s}$, the one-step approximation is symplectic which will be proved in Section \ref{4}.

Approximating It$\hat{o}$ integral in (\ref{6.2}) by the left-rectangle formula and the Pad$\acute e$ approximation $P_{(1,1)}(hJ^{-1}\tilde{C}^0)$, we obtain the explicit numerical scheme
\begin{equation}\label{6.4}
\begin{split}
Z_{n+1}=&\left[I+\sum_{j=1}^{\hat{s}} b_j(-B_1)^{j}\right]^{-1}\left[I+\sum_{j=1}^{\hat{r}} a_j(B_1)^{j}\right]Z_{n}\\
&+\sum_{i=1}^m\left[I-\frac{1}{2}B_1\right]^{-1}\left[I+\frac{1}{2}B_1\right]J^{-1}R_{i}\Delta W^{i}_{n},
\end{split}
\end{equation}
with every $\hat{s}$, $\hat{r}$ $\geq 1$. Similar to the special case of the numerical scheme (\ref{6.3}),  the one-step approximation (\ref{6.4}) preserves symplecticity if $\hat{r}=\hat{s}$. Moreover, we can prove that the numerical method (\ref{6.4}) is of mean-square order 1.

\section{Properties of the numerical methods based on the Pad$\acute{e}$ approximation}
\label{4}
In this section, we prove the mean-square convergence order of the proposed schemes $(\ref{3.12})\sim(\ref{3.15})$ and (\ref{6.3}) based on the Pad$\acute{e}$ approximation, as well as their symplecticity. For the former purpose, we need the fundamental convergence theorem in stochastic context given by G. N. Milstein and M.V. Tretyakov(\cite{mil1,mil2,milbook}).
\begin{Proposition}\label{th1}
(see \cite{milbook,mil1,mil2}) Suppose the one-step approximation $\bar{X}(t+h;t,x)$ has the order of accuracy $p_{1}$ for the mathematical expectation of the deviation and order of accuracy $p_{2}$ for the mean-square deviation; more precisely, for arbitrary $t_{0}\leq t\leq t_{0}+T-h$, $x\in R^{n}$ the following inequalities hold:
\begin{equation}\label{3.1}
\begin{array}{l}
  |{\bf {E}}(X(t+h;t,x)-\bar{X}(t+h;t,x))|\leq K{(1+|x|^2)}^{\frac{1}{2}}h^{p_1},\\
  {\Bigl[{\bf {E}}{|X(t+h;t,x)-\bar{X}(t+h;t,x)|}^{2}\Bigr]}^{\frac{1}{2}}\leq K{(1+|x|^2)}^{\frac{1}{2}}h^{p_2},
\end{array}
\end{equation}
Also, let
\begin{equation*}\label{3.2}
{p_{2}}\geq {\frac{1}{2}}, {p_{1}}\geq {{p_{2}}+{\frac{1}{2}}},
\end{equation*}
Then for any $N$ and $k=0,\cdots,N$ the following inequality holds:
\begin{equation}\label{3.3}
{\Bigl[{\bf {E}}{|X(t_{k};t_0,X_0)-\bar{X}(t_{k};t_0,X_0)|}^{2}\Bigr]}^{\frac{1}{2}}\leq K{(1+|X_0|^2)}^{\frac{1}{2}}h^{p_2-\frac{1}{2}},
\end{equation}
i.e. the order of accuracy of the method constructed using the one-step approximation $\bar{X}(t+h;t,x)$ is $p=p_2-\frac{1}{2}$.
\end{Proposition}

We note that all the constants $K$ mentioned above, as well as the ones that will appear in the sequels, depend on the system and the approximation only and do not depend on $X_0$ and $N$. 
\begin{Proposition}\label{th2}
(see \cite{milbook,mil1,mil2})Let the one-step approximation $\bar{X}(t+h;t,x)$ satisfy the condition of Theorem \ref{th1}. Suppose that ${\tilde{X}}(t+h;t,x)$ is such that
\begin{equation}\label{3.4}
\begin{array}{l}
  |{\bf {E}}(\bar{X}(t+h;t,x)-\tilde{X}(t+h;t,x))|= O(h^{p_1}),\\
  {\Bigl[{\bf {E}}{|\bar{X}(t+h;t,x)-\tilde{X}(t+h;t,x)|}^{2}\Bigr]}^{\frac{1}{2}}= O(h^{p_2}),
\end{array}
\end{equation}
with the same $h^{p_1}$ and $h^{p_2}$. Then the method based on the one-step approximation ${\tilde{X}}(t+h;t,x)$ has the same mean-square order of accuracy as the method based on ${\bar{X}}(t+h;t,x)$, i.e., its order is equal to $p=p_2-\frac{1}{2}$.
\end{Proposition}

Our result regarding the mean-square convergence order of the schemes $(\ref{3.12})\sim(\ref{3.15})$ and (\ref{6.3}) is as follows.
\begin{theorem}\label{th3}
The numerical methods $X^{(r,s)}_{n}$ e.g. (\ref{3.12})$\sim$(\ref{3.15}) based on the Pad$\acute{e}$ approximation $P_{(r,s)}$ with $A_{h}^{i}=\sqrt{2\ell|\ln h|}$, $\ell \geq r+s$ is of mean-square order $\frac{r+s}{2}$.
\end{theorem}
\begin{proof}
According to an analog of the Taylor expansion of the solution $X(t+h;t,x)$ in (\ref{2.2}), we obtain a one-step approximation $\hat {X}(t+h;t,x)$ as follows
\begin{equation}
\label{4.1}
\begin{split}
\hat X&=x+\sum_{j=1}^{r+s} \frac{1}{j!}(hA^0+\sum_{i=1}^{m}\sqrt h  \zeta_{h}^{i}A^{i})^{j}x\\
&=x+\sum_{j=1}^{r+s} \frac{1}{j!}\bar B^{j}x,
\end{split}
\end{equation}
which has the $\frac{r+s}{2}$-$th$ mean-square order of convergence according to the Proposition \ref{th1} by verifying the $|{\bf {E}}(X-\hat X)|$ and ${\bf {E}}|(X-\hat X)|^2$.
First,
\begin{equation*}
\label{4.1051}
\begin{split}
    &|{\bf {E}}(X-\hat X)|\\
    \leq &K|\sum_{j=1}^{r+s} \frac{1}{j!}[{\bf {E}}(hA^0+\sum_{i=1}^{m}\sqrt h  \xi^{i}A^{i})^{j}-{\bf {E}}(hA^0+\sum_{i=1}^{m}\sqrt h  \zeta_{h}^{i}A^{i})^{j}]|\\
    +&K|\sum_{j=r+s+1}^{+\infty}\frac{1}{j!}{\bf {E}}(hA^0+\sum_{i=1}^{m}\sqrt h \xi^{i}A^{i})^{j}|\\
    \leq &K\sum_{j=1}^{r+s}\sum_{i=1}^{m} \frac{1}{j!}\sum_{\rho=0}^{j}h^{j-\rho+\frac{\rho}{2}}|{\bf {E}}({\xi^{i}}^\rho-{\zeta_{h}^{i}}^\rho)|+O(h^{\left[\frac{r+s}{2}\right]+1}).
\end{split}
\end{equation*}
According to the distribution function of the random variable $\xi^{i}$, we have
\begin{equation*}
\begin{split}
    &|{\bf {E}}(X-\hat X)|\\
    \leq&K\sum_{j=1}^{r+s}\sum_{i=1}^{m} \frac{1}{j!}\sum_{\rho=even}^{j}h^{j-\rho+\frac{\rho}{2}}| \int_{A_h}^{+\infty}(x^\rho-A_{h}^\rho)\exp (-\frac{x^2}{2})dx|+O(h^{\left[\frac{r+s}{2}\right]}+1)\\
    \leq&K\sum_{j=1}^{r+s} \frac{1}{j!}\sum_{\rho=even,\geq 2}^{j}h^{j-\rho+\frac{\rho}{2}}\sum_{u=1}^{\rho-1}A_h^u\exp (-\frac{A_h^2}{2})\int_{0}^{+\infty}x^{\rho-u}\exp(-\frac{x^2}{2})dx+O(h^{\left[\frac{r+s}{2}\right]+1}).
\end{split}
\end{equation*}
From the condition $A_{h}^{2}\geq 2\ell|\ln h|$ which implies $\exp (-\frac{A_h^2}{2})\leq h^{l}$, we can get
\begin{equation*}
    |{\bf {E}}(X-\hat X)|
    \leq K\sum_{j=1}^{r+s} \frac{1}{j!}\sum_{\rho=even}^{j}h^{j-\rho+\frac{1}{2}+\ell}+O(h^{\left[\frac{r+s}{2}\right]+1})
    =O(h^{\left[\frac{r+s}{2}\right]+1}).
\end{equation*}
Secondly, we estimate ${\bf {E}}|(X-\hat X)|^2$,
\begin{equation}
\label{4.1052}
\begin{split}
    &{\bf {E}}|(X-\hat X)|^2\\
    \leq &K{\bf {E}}|\sum_{j=1}^{r+s} \frac{1}{j!}[(hA^0+\sum_{i=1}^{m}\sqrt h  \xi^{i}A^{i})^{j}-(hA^0+\sum_{i=1}^{m}\sqrt h  \zeta_{h}^{i}A^{i})^{j}]|^2\\
    +&K{\bf {E}}|\sum_{j=r+s+1}^{+\infty}\frac{1}{j!}(hA^0+\sum_{i=1}^{m}\sqrt h \xi^{i}A^{i})^{j}|^2\\
    \leq &K\sum_{j=1}^{r+s}\sum_{i=1}^{m} \frac{1}{j!}\sum_{\rho=0}^{j}h^{2j-\rho}{\bf {E}}|({\xi^{i}}^\rho-{\zeta_{h}^{i}}^\rho)|^2+O(h^{r+s+1})\\
    \leq&K\sum_{j=1}^{r+s}\sum_{i=1}^{m} \frac{1}{j!}\sum_{\rho=even}^{j}h^{2j-\rho}| \int_{A_h}^{+\infty}(x^\rho-A_{h}^\rho)^2\exp (-\frac{x^2}{2})dx|+O(h^{r+s+1})\\
    \leq&K\sum_{j=1}^{r+s} \frac{1}{j!}\sum_{\rho=even,\geq 0}^{j}h^{2j-\rho}\sum_{u=1}^{\rho-1}A_h^{2u}\exp (-\frac{A_h^2}{2})\int_{0}^{+\infty}x^{2\rho-2u}\exp(-\frac{x^2}{2})dx\\
    +&O(h^{r+s+1})\\
    \leq&K\sum_{j=1}^{r+s} \frac{1}{j!}\sum_{\rho=even}^{j}h^{2j-2\rho+1+\ell}+O(h^{r+s+1})\\
    =&O(h^{r+s+1}),
\end{split}
\end{equation}
Thus, the numerical scheme (\ref{4.1}) is of the mean-square order $\frac{r+s}{2}$. Recall the one-step approximation $X^{(r,s)}(t+h;t,x)$ (\ref{3.11}) we proposed
\begin{equation}
\label{4.2.5}
      X^{(r,s)}={\left[I+\sum_{j=1}^{s} b_j(-\bar B)^{j}\right]}^{-1}\left[I+\sum_{j=1}^{r} a_j\bar B^j\right]x,
\end{equation}
where $\bar B=hA^0+{\mathbf{\sum}}_{i=1}^{m}\sqrt h  \zeta_{h}^{i}A^{i}$. From (\ref{3.8}), we know that
\begin{equation}
\label{4.3}
\begin{split}
     \hat X-X^{(r,s)}&=x+\sum_{j=1}^{r+s} \frac{1}{j!}\bar B^{j}x-P_{(r,s)}(\bar B)x\\
     &=exp(\bar B)x-P_{(r,s)}(\bar B)x-\sum_{j=r+s+1}^{+\infty} \frac{1}{j!}\bar B^{j}x\\
     &=\sum_{j=r+s+1}^{+\infty} \bar c_{j}\bar B^{j}x,
\end{split}
\end{equation}
where $\bar c_{j}$, $j\geq r+s+1$, are constants. It is obvious that
\begin{equation}
\label{4.4}
    \begin{split}
    |{\bf {E}}(\hat X-X^{(r,s)})|
    &\leq K_1|{\bf {E}}\sum_{j=r+s+1}^{+\infty}(hA^0+{\mathbf{\sum}}_{i=1}^{m}\sqrt h  \zeta_{h}^{i}A^{i})^{j}|\\
    &=O(h^{\left[\frac{r+s}{2}\right]+1}),
    \end{split}
\end{equation}
and
\begin{equation}
\label{4.5}
    \begin{split}
    {{\bf {E}}|\hat X-\bar{X}^{(r,s)}|^{2}}
    &\leq K_1{\bf {E}}|\sum_{j=r+s+1}^{+\infty}(hA^0+{\mathbf{\sum}}_{i=1}^{m}\sqrt h\zeta_{h}^{i}A^{i})^{j}|^{2}\\
    &=O(h^{r+s+1}),
    \end{split}
\end{equation}
where $K_1$ is a sufficiently large constant. Applying Proposition \ref{th2}, we prove the theorem.\hspace*{1cm}
\end{proof}

\begin{theorem}
\label{th5}
The numerical method (\ref{6.3}) based on the Pad$\acute{e}$ approximation $P_{(\hat{r},\hat{s})}$ and $P_{(\check r,\check s)}$ as follows
\begin{equation*}
\begin{split}
Z_{n+1}=&\left[I+\sum_{j=1}^{\hat{s}} b_j(-B_1)^{j}\right]^{-1}\left[I+\sum_{j=1}^{\hat{r}} a_j(B_1)^{j}\right]Z_{n}\\
&+\sum_{i=1}^m \int_{t_{n}}^{t_{n+1}}\left[I+\sum_{j=1}^{\check{s}} b_j(-B_2)^{j}\right]^{-1}\left[I+\sum_{j=1}^{\check{r}} a_j(B_2)^{j}\right]J^{-1}R_{i}dW^{i}(\theta),
\end{split}
\end{equation*}
where $B_1=hJ^{-1}\tilde{C}^{0}$ and $B_2=(t_{n+1}-\theta)J^{-1}\tilde{C}^{0}$ with $\hat{r}+\hat{s}=\check r+\check s+2$ and $\check r, \check s\geq 1$ is of mean-square order $\check r+\check s+1$.
\end{theorem}

\begin{proof}
It is known that the solution $Z(t+h;t,z)$ of (\ref{6.1}) is
\begin{equation}\label{7.1}
Z=e
^{hJ^{-1}\tilde{C}^{0}}z+\sum_{i=1}^m \int_{t}^{t+h}e^{(t+h-\theta)J^{-1}\tilde{C}^{0}}J^{-1}R_{i}dW^{i}(\theta),
\end{equation}
and the one-step approximation $\bar Z(t+h;t,z)$ (\ref{6.3}) is as follows
\begin{equation}
\label{7.2}
\bar Z=P_{\hat r,\hat s}(hJ^{-1}\tilde{C}^{0})z+\sum_{i=1}^m \int_{t}^{t+h}P_{\check r,\check s}((t+h-\theta)J^{-1}\tilde{C}^{0})J^{-1}R_{i}dW^{i}(\theta).
\end{equation}
By estimating  $|{\bf {E}}(Z-\bar Z)|$ and ${\bf {E}}|(Z-\bar Z)|^2$, we know that
\begin{equation}
\label{7.3}
    |{\bf {E}}(Z-\bar Z)|\leq K|e^{hJ^{-1}\tilde{C}^{0}}-P_{\hat r,\hat s}(hJ^{-1}\tilde{C}^{0})|=O(h^{\hat r+\hat s+1}).
\end{equation}
and
\begin{equation}
\label{7.4}
\begin{split}
    &{\bf {E}}|(Z-\bar Z)|^2\\
    \leq &K{\bf {E}}|e^{hJ^{-1}\tilde{C}^{0}}-P_{\hat r,\hat s}(hJ^{-1}\tilde{C}^{0})|^2\\
    +&K{\bf {E}}|\sum_{i=1}^m \int_{t}^{t+h}[e^{(t+h-\theta)J^{-1}\tilde{C}^{0}}-P_{\check r,\check s}((t+h-\theta)J^{-1}\tilde{C}^{0})]J^{-1}R_{i}dW^{i}(\theta)|^2.
\end{split}
\end{equation}
Due to the independence of the Winner processes and the isometry's property of the It$\hat o$ integral, the equation (\ref{7.4}) becomes
\begin{equation}
\label{7.5}
\begin{split}
    &{\bf {E}}|(Z-\bar Z)|^2\\
    \leq &K\sum_{i=1}^m \int_{t}^{t+h}{\bf {E}}| [e^{(t+h-\theta)J^{-1}\tilde{C}^{0}}-P_{\check r,\check s}((t+h-\theta)J^{-1}\tilde{C}^{0})]J^{-1}R_{i}|^2d\theta\\
    +&O(h^{2\hat{r}+2\hat{s}+2})\\
    \leq &K\sum_{i=1}^m \int_{t}^{t+h}{\bf {E}}| (t+h-\theta)^{\check r+\check s+1}|^2d\theta+O(h^{2\check{r}+2\check{s}+5})+O(h^{2\hat{r}+2\hat{s}+2})\\
    =&O(h^{2\check{r}+2\check{s}+3})
\end{split}
\end{equation}
Thus, according to Proposition \ref{th1}, the numerical scheme (\ref{6.3}) is of the mean-square order $\check r+\check s+1$.
\end{proof}

Now let us prove the symplecticity of the schemes
(\ref{3.12})$\sim$(\ref{3.15}) and (\ref{6.3}).

\begin{theorem}\label{th4}
The numerical methods $X^{(k,k)}_{n}$ e.g. (\ref{3.12})$\sim$(\ref{3.15}) based on the Pad$\acute{e}$ approximation are symplectic.
\end{theorem}
\begin{proof}
The proof of this theorem is a straight-forward extension of its counterpart in deterministic case in (\cite{feng}), and there is no essential difficulty arising from involving stochastic elements.

For the sake of simplicity, we also denote $\bar B=hA^0+{\mathbf{\sum}}_{i=1}^{m}\sqrt h\zeta_{h}^{i}A^{i}$ as above and consider the one-step approximation (\ref{3.11}) based on $P_{(k,k)}(\bar B)$ with $A_{h}^{i}=\sqrt{4k|\ln h|}$ as follows,
\begin{equation}
X_{n+1}^{(k,k)}={
\left[D_{(k,k)}(\bar B)\right]}^{-1}{\left[N_{(k,k)}(\bar B)\right]}X_{n}^{(k,k)}.
\end{equation}
where $$D_{(k,k)}(\bar B)=I+\sum_{j=1}^{k} a_j(-hA^0-{\mathbf{\sum}}_{i=1}^{m}\sqrt h\zeta_{h}^{i}A^{i})^{j},$$
and
$$N_{(k,k)}(\bar B)=I+\sum_{j=1}^{k} a_j(hA^0+{\mathbf{\sum}}_{i=1}^{m}\sqrt h\zeta_{h}^{i}A^{i})^{j}.$$
Let $N_{(k,k)}(\bar B)=F(\bar B)+G(\bar B)$, $D_{(k,k)}(\bar B)=F(\bar B)-G(\bar B)$, where $F(\bar B)$ is an even polynomial and $G(\bar B)$ is an odd polynomial. Since $\bar B\in sp(2n)$ , we get $F(\bar B^{\top})J=JF(\bar B)$ and $G(\bar B^{\top})J+JG(\bar B)=0$ according to the Theorem \ref{th0.51} and \ref{th0.52}. We want to verify the numerical method $X_{n+1}^{(k,k)}$ is symplectic, so we should prove that the matrix $D_{(k,k)}^{-1}N_{(k,k)}\in Sp(2n)$. Since
\begin{equation}
\label{4.6}
\begin{split}
     &N_{(k,k)}^{\top}JN_{(k,k)}\\
     =&(F(\bar B^{\top})+G(\bar B^{\top}))J(F(\bar B)+G(\bar B))\\
     =&J(F(\bar B)-G(\bar B))(F(\bar B)+G(\bar B))\\
     =&J(F(\bar B)+G(\bar B))(F(\bar B)-G(\bar B))\\
     =&(F(\bar B^{\top})-G(\bar B^{\top}))J(F(\bar B)-G(\bar B))\\
     =&D_{(k,k)}^{\top}JD_{(k,k)},
\end{split}
\end{equation}
From the Theorem \ref{th0.53} we know that  $D_{(k,k)}^{-1}N_{(k,k)}\in Sp(2n)$. Thus we prove the theorem.
\end{proof}

Using arguments similar to ones in the proof of Theorem {\ref{th4}}, we obtain the following theorem.

\begin{theorem}\label{th6}
If $\hat{r}$ equals $\hat{s}$, the numerical methods $Z_{n+1}$ (\ref{7.6}) as follows
\begin{equation}
\label{7.6}
Z_{n+1}=P_{\hat r,\hat r}(hJ^{-1}\tilde{C}^{0})Z_{n}+\sum_{i=1}^m \int_{t_{n}}^{t_{n}+h}P_{\check r,\check s}((t_{n}+h-\theta)J^{-1}\tilde{C}^{0})J^{-1}R_{i}dW^{i}(\theta),
\end{equation}
is symplectic.
\end{theorem}
\begin{proof}
It is known that the one-step approximation (\ref{7.6})
is symplectic iff $${\frac{\partial Z_{n+1}}{\partial Z_{n}}}^{\top}J\frac{\partial Z_{n+1}}{\partial Z_{n}}=J.$$
Since $\frac{\partial Z_{n+1}}{\partial Z_{n}}=P_{\acute r,\acute r}(hJ^{-1}\tilde{C}^{0})$, the numerical method (\ref{7.6}) is symplectic iff $P_{\acute r,\acute r}\in Sp(2n)$. Repeating the proof as that of the Theorem \ref{th4}, we prove the Theorem.
\end{proof}
\section{Numerical tests}
\label{5}
\noindent{{\bf {Example 1}}}\quad We consider the system of SDEs in the sense of Stratonovich, i.e. the Kubo oscillator
\begin{equation}\label{5.1}
\begin{split}
&dP=-aQdt-\sigma Q\circ dW(t),\quad P(0)=p,\\
&dQ=aPdt+\sigma P\circ dW(t),\quad Q(0)=q,
\end{split}
\end{equation}
where a and $\sigma$ are constants and $W(t)$ is a one-dimensional standard Wiener process. The exact solution of (\ref{5.1}) is
\begin{equation}
\label{5.2}
\begin{split}
&P(t)=pcos(at+\sigma W(t))-qsin(at+\sigma W(t)),\\
&Q(t)=psin(at+\sigma W(t))+qcos(at+\sigma W(t)).
\end{split}
\end{equation}
The phase flow of this system preserves symplectic structure. Moreover, the quantity $H(p,q) = p^2+q^2,$ is conservative for this system, i.e.:
$$H(P(t),Q(t))= H(p,q),\quad \mbox{for}\quad t\geq0.$$
This means that the phase trajectory of (\ref{5.1}) is a circle centered at the origin.
\begin{figure*}[t]
\centering
\subfigure{
\includegraphics[width=6cm,height=5cm]{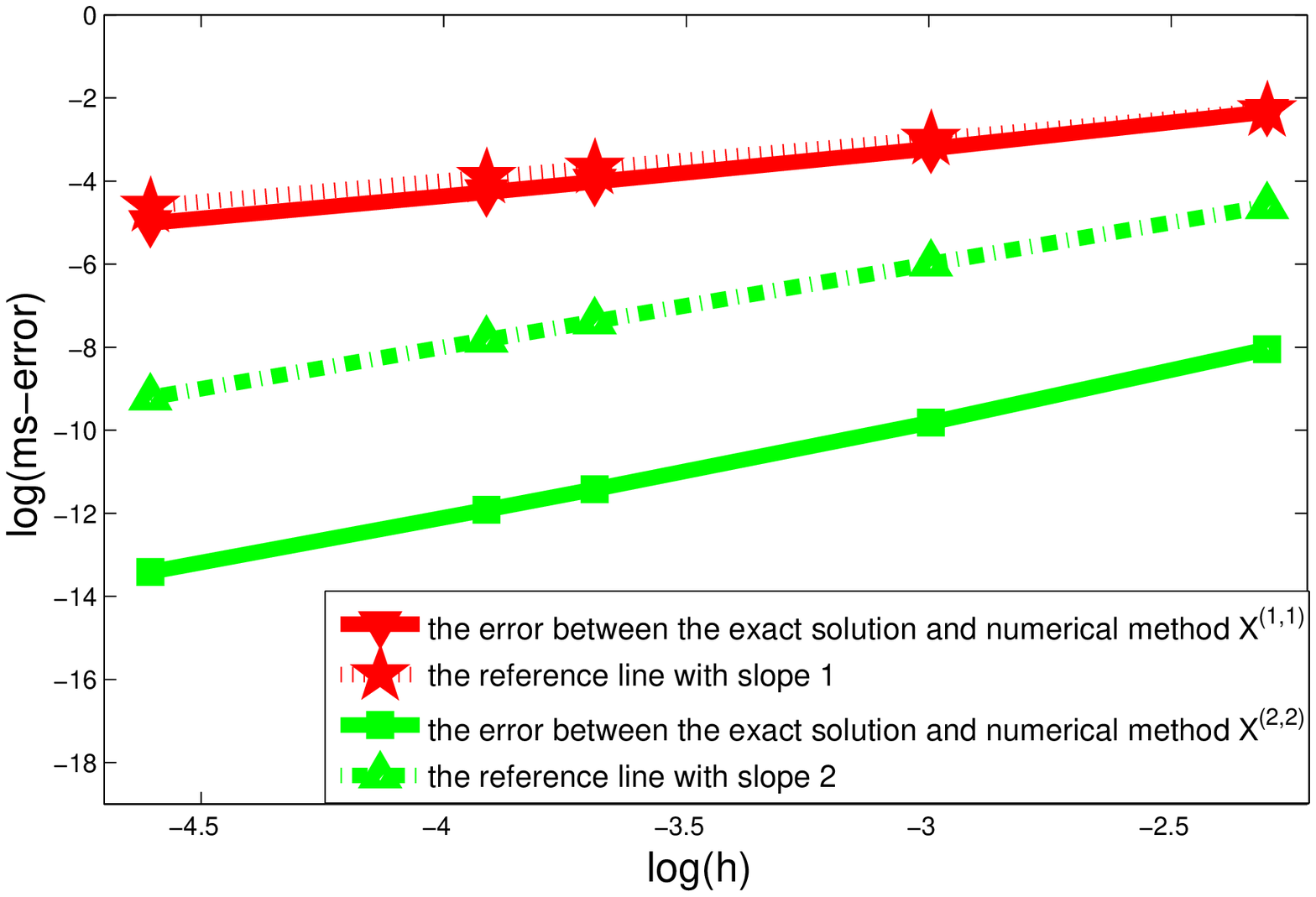}}
\subfigure{
\includegraphics[width=6cm,height=5cm]{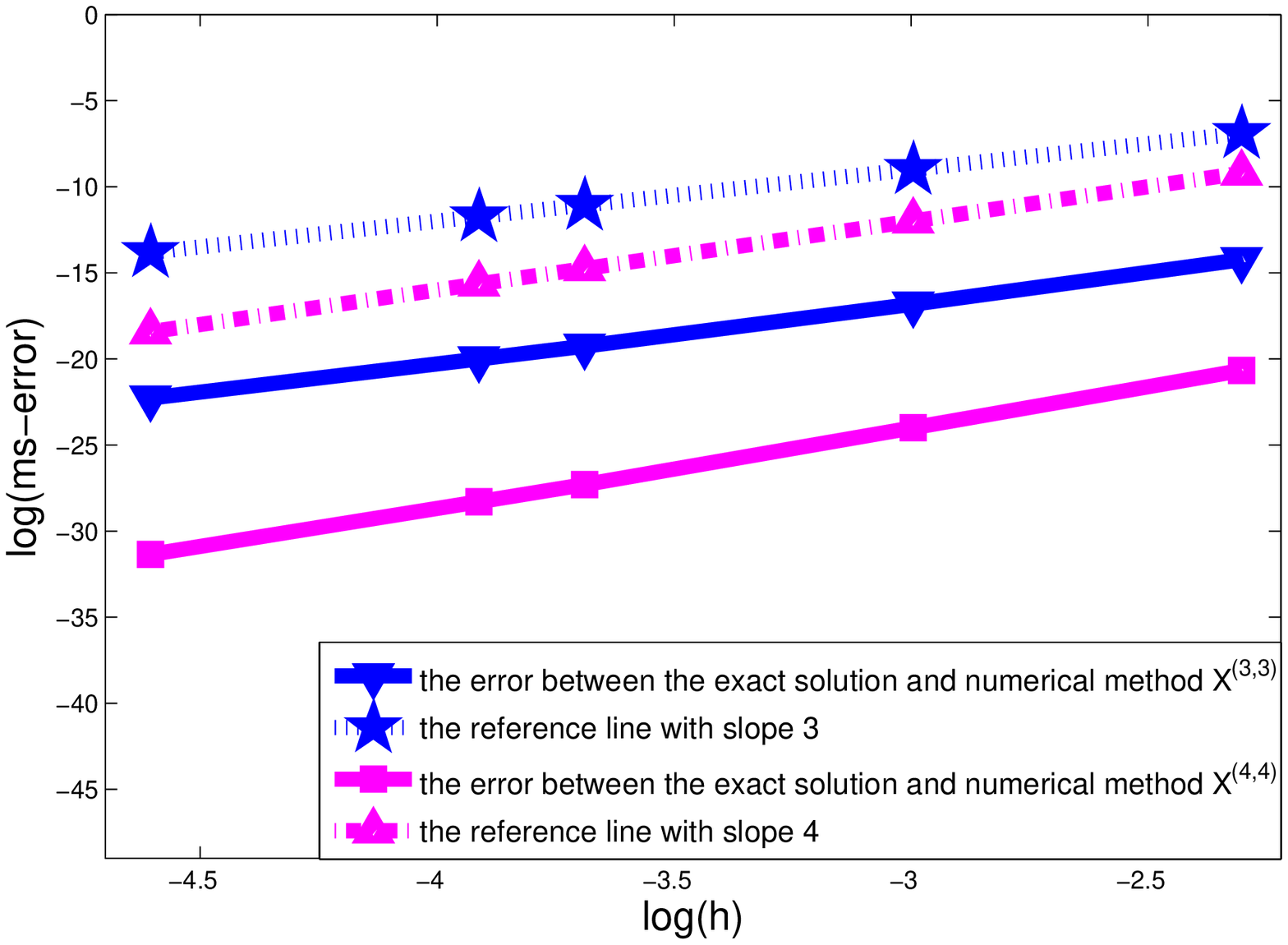}}
\caption{The mean-square convergence order of the scheme (\ref{5.3}) and (\ref{5.4}) (left), the mean-square convergence order of the scheme (\ref{5.5}) and (\ref{5.55}) (right).}\label{pp1}
\end{figure*}

\begin{figure*}[t]
\centering
\subfigure{
\includegraphics[width=6cm,height=5cm]{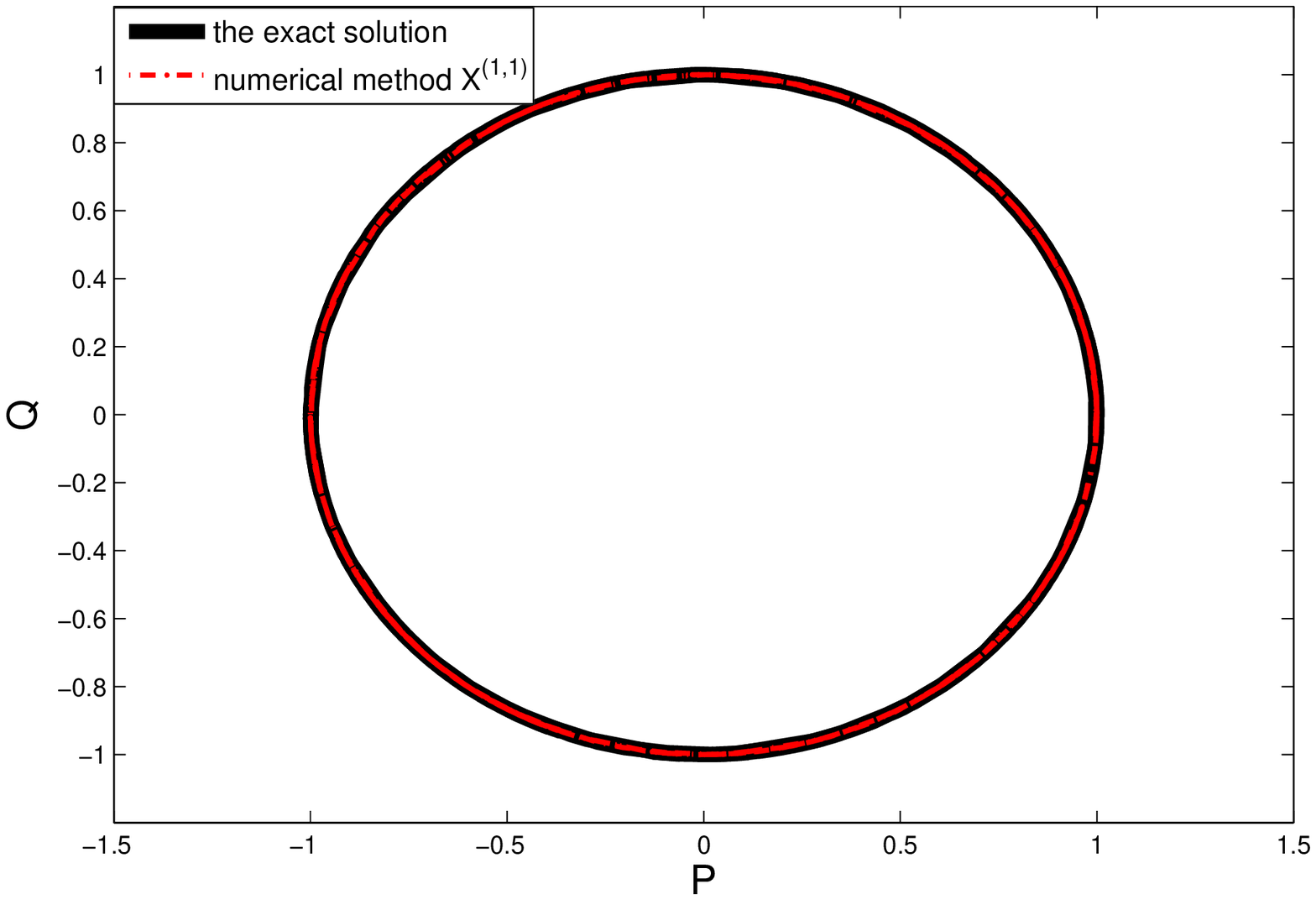}}
\subfigure{
\includegraphics[width=6cm,height=5cm]{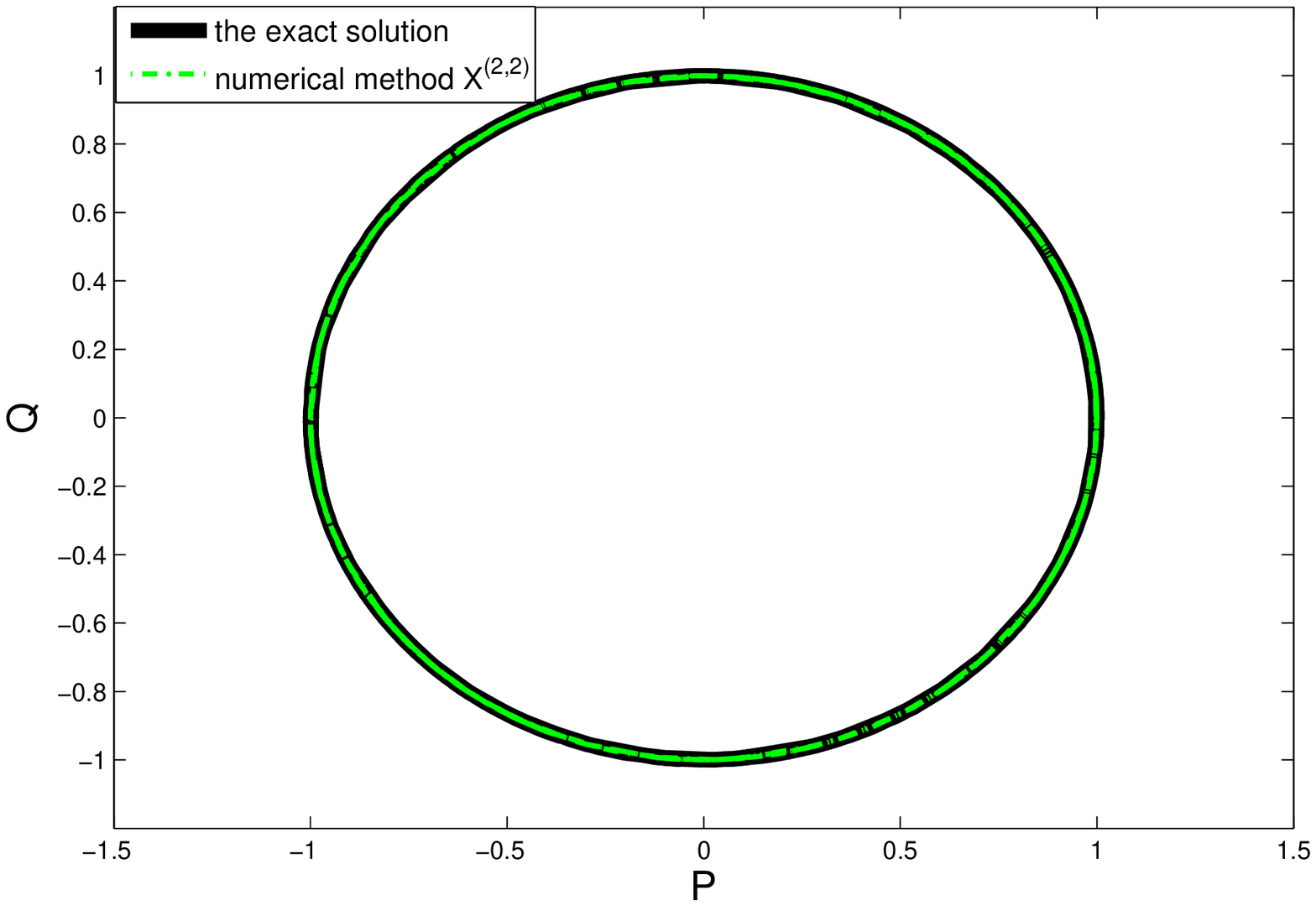}}
\subfigure{
\includegraphics[width=6cm,height=5cm]{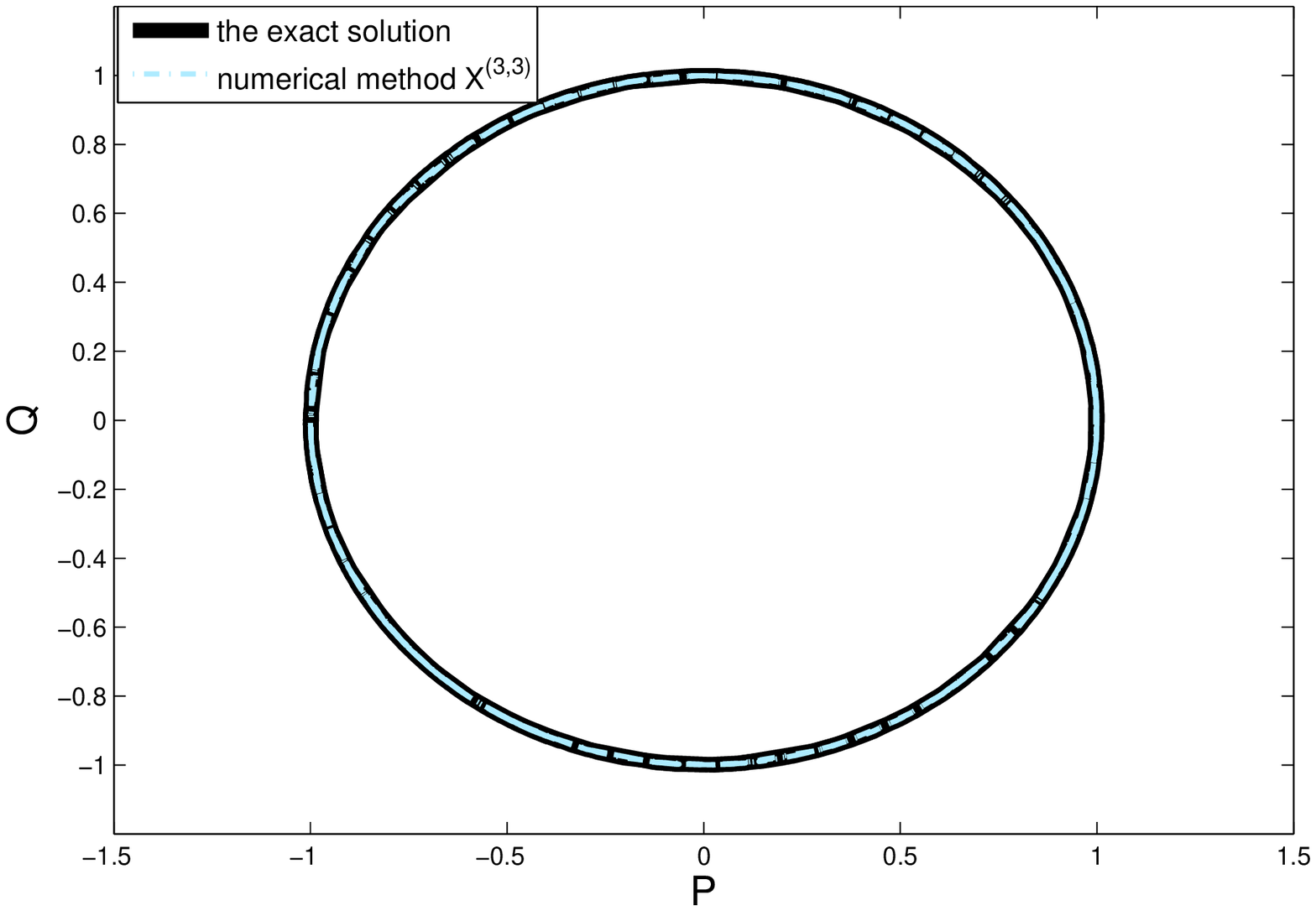}}
\subfigure{
\includegraphics[width=6cm,height=5cm]{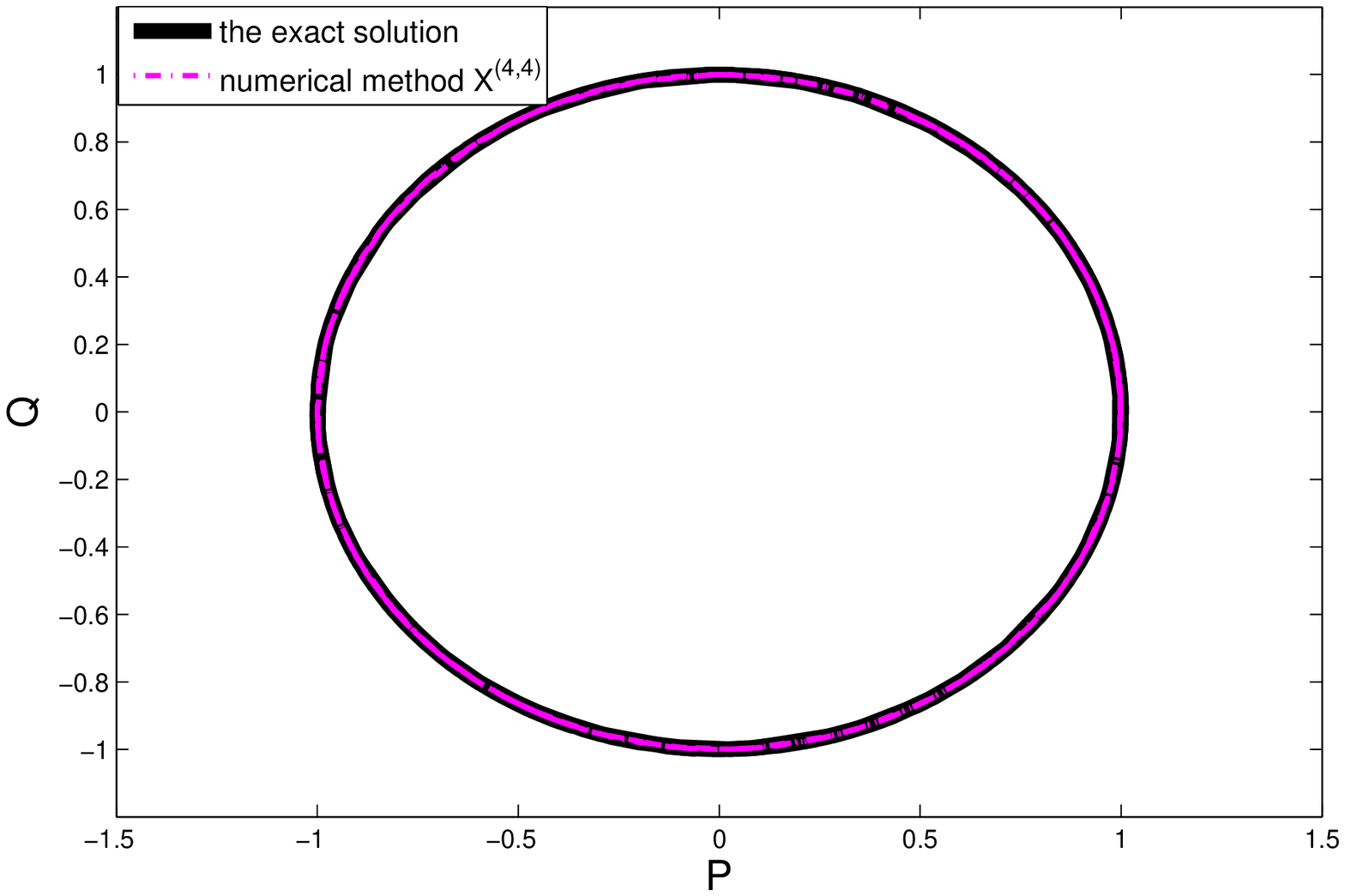}}
\caption{A sample path trajectory of the scheme (\ref{5.3}) and (\ref{5.4}) (above), A sample path trajectory of the scheme (\ref{5.5}) and (\ref{5.55}) (below).}\label{pp2}
\end{figure*}

We test here four methods. Denoting by $\bar B_{i}=(ah+\sigma \sqrt h\zeta_{i})$, $i=1,2,3,4$, the symplectic methods
(\ref{3.12})$\sim$(\ref{3.15}) applied to (\ref{5.1}) take the following forms, respectively,
\begin{equation}
\label{5.3}
\bar{X}_{n+1}^{(1,1)}=\bar{X}_{n}^{(1,1)}+\frac{1}{2} \bar B_{1}(\bar{X}_{n}^{(1,1)}+\bar{X}_{n+1}^{(1,1)}),
\end{equation}
with $A_{h}^{1}=\sqrt{4|\ln h|}$,
\begin{equation}
\label{5.4}
\bar{X}_{n+1}^{(2,2)}=\bar{X}_{n}^{(2,2)}+\frac{1}{2}\bar B_{2}(\bar{X}_{n}^{(2,2)}+\bar{X}_{n+1}^{(2,2)})+\frac{1}{12}\bar B_{2}^2(\bar{X}_{n}^{(2,2)}-\bar{X}_{n+1}^{(2,2)}),
\end{equation}
with $A_{h}^{2}=\sqrt{8|\ln h|}$,
\begin{equation}
\label{5.5}
\bar{X}_{n+1}^{(3,3)}=\bar{X}_{n}^{(3,3)}+(\frac{1}{2}\bar B_{3}+\frac{1}{120}\bar B_{3}^3)(\bar{X}_{n}^{(3,3)}+\bar{X}_{n+1}^{(3,3)})+\frac{1}{10}\bar B_{3}^2(\bar{X}_{n}^{(3,3)}-\bar{X}_{n+1}^{(3,3)}),
\end{equation}
with $A_{h}^{2}=\sqrt{12|\ln h|}$, and
\begin{equation}
\label{5.55}
\bar{X}_{n+1}^{(4,4)}=\bar{X}_{n}^{(4,4)}+(\frac{1}{2}\bar B_{4}+\frac{1}{84}\bar B_{4}^3)(\bar{X}_{n}^{(4,4)}+\bar{X}_{n+1}^{(4,4)})+(\frac{1}{24}\bar B_{4}^2+\frac{1}{1680}\bar B_{4}^4)(\bar{X}_{n}^{(4,4)}-\bar{X}_{n+1}^{(4,4)}).
\end{equation}
with $A_{h}^{2}=\sqrt{16|\ln h|}$.


The numerical tests examine the behaviors of the numerical methods from three aspects: first, the convergence rate of the numerical methods illustrated by Figure 1; second, the sample trajectory produced by the numerical methods and the true solution, as shown by Figure 2; and third, the Hamiltonians produced by the numerical methods, as presented in Figure 3.

Figure \ref{pp1} shows that, comparing with the reference lines, the numerical method (\ref{5.3}) is of mean-square order 1, the numerical schemes (\ref{5.4}) and (\ref{5.5}) are of the second and third mean square order respectively, and the numerical scheme (\ref{5.6}) is of mean-square order 4. These validate the theorem regarding mean-square convergence order of the proposed methods. In our experiments we take $T=5$, $p=1$, $q=0$ and $h=[0.01,0.02,0.025,0.05,0.1]$. The expectation $\bf{E}$ is approximated by taking average over 1000 sample paths.

Figure \ref{pp2} gives approximations of a sample phase trajectory of (\ref{5.1}) simulated by the symplectic methods (\ref{5.3})$\sim$(\ref{5.6}). The initial condition is $p=1$, $q=0$. The corresponding exact phase trajectory belongs to the circle with the center at the origin and with the unit radius. We see that the symplectic methods are appropriate for simulation of the oscillator (\ref{5.1}) on long time intervals $[0,100]$.

\begin{figure*}[t]
\centering\subfigure{
\includegraphics[width=6cm,height=5cm]{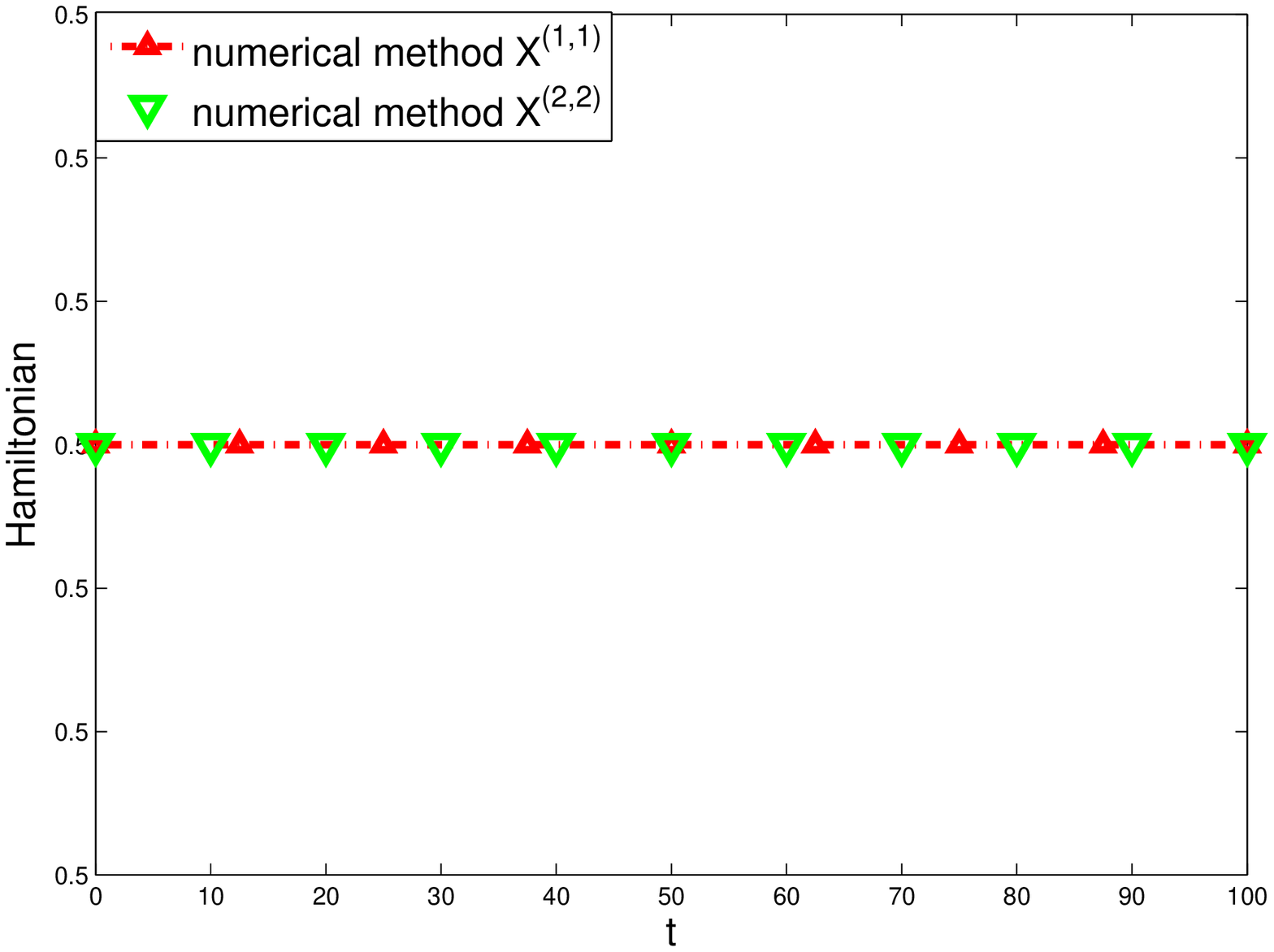}}
\subfigure{
\includegraphics[width=6cm,height=5cm]{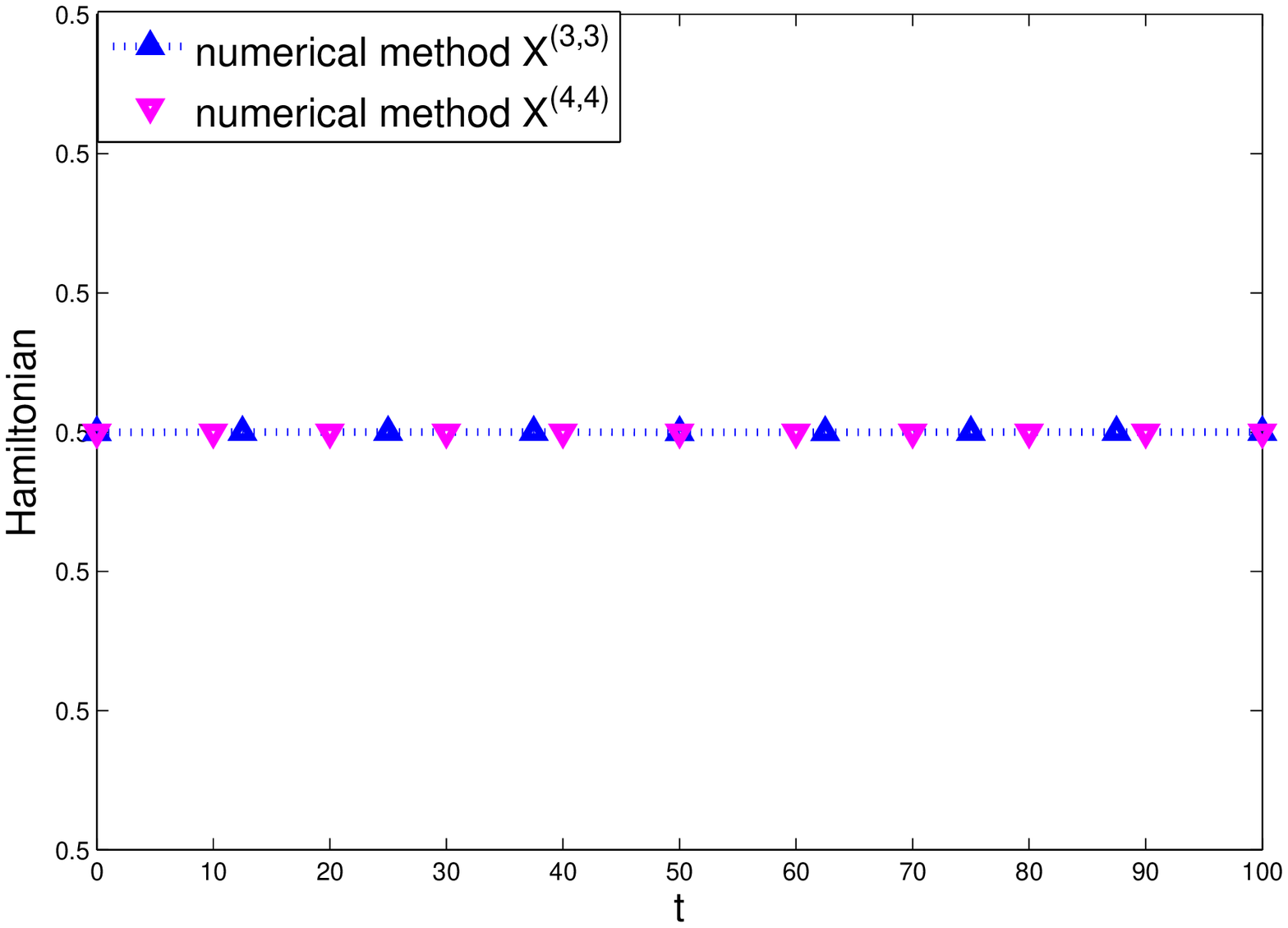}}
\caption{The conservation of the Hamiltonian of the scheme (\ref{5.3}) and (\ref{5.4}) (left), The conservation of the Hamiltonian of the scheme (\ref{5.5}) and (\ref{5.55}) (right)..}\label{pp3}
\end{figure*}

It is not difficult to check that $H(P, Q)$ is conserved by the numerical methods (\ref{5.3})$\sim$(\ref{5.6}). Figure \ref{pp3} illustrate this fact as we can see the Hamiltonian of the numerical schemes (\ref{5.3})$\sim$(\ref{5.6}) do not change. And we take $T=100$, $p=1$, $q=0$ in our experiments.

\vspace{1cm}
\noindent{{\bf Example 2}}\quad A linear stochastic oscillator.

The linear stochastic oscillator (\cite{melbo})
\begin{equation}
\label{5.6}
\begin{array}{l}
dp=-qdt+\sigma dW(t),\quad p(0)=0,\\
dq=pdt,\quad q(0)=1
\end{array}
\end{equation}
is a stochastic Hamiltonian system with $H=\frac{1}{2}(p^2+q^2)$, $H_1=-\sigma q$. Given the initial values, it can be shown (\cite{melbo}) that the system (\ref{5.6}) has the exact solution
\begin{equation}
\label{5.7}
\begin{split}
p(t)&=\quad p_0\cos t+q_0\sin t+\sigma\int_{0}^{t}\sin (t-s)dW(s),\\
q(t)&=-p_0\sin t+q_0\cos t+\sigma\int_{0}^{t}\cos (t-s)dW(s),
\end{split}
\end{equation}
which possess the two properties,\\
(1)the second moment ${\bf{E}}(p(t)^2+q(t)^2)=1+\sigma^2t$;\\
(2)(Markus and Weerasinghe \cite{Markus}) $p(t)$ has infinitely many zeros, all simple, on each half line $[t_0,\infty)$ for every $t_0\geq 0$, a.s..

Regarding the one-step approximations we proposed for the stochastic differential equations (\ref{5.6}), if we replace the $P_{(\hat{r},\hat{s})}$ and $P_{(\check{r},\check{s})}$ in (\ref{6.3}) with $P_{(2,2)}$ and $P_{(1,1)}$ respectively we will get
\begin{equation}
\label{5.8}
\begin{split}
Z_{n+1}=&\left[I-\frac{h}{2}J+\frac{h^2}{12}J^2\right]^{-1}\left[I+ \frac{h}{2}J+\frac{h^2}{12}J^2\right]Z_{n}\\
&+\int_{t_{n}}^{t_{n+1}}\left[I-\frac{t_{n+1}-\theta}{2}J\right]^{-1}\left[I+\frac{t_{n+1}-\theta}{2}J\right]\begin{pmatrix}
\sigma \\ 0
\end{pmatrix}dW(\theta),
\end{split}
\end{equation}
and if both $P_{(\hat{r},\hat{s})}$ and $P_{(\check{r},\check{s})}$ in (\ref{6.4}) are taken place by $P_{(1,1)}$, the numerical methods as follows
\begin{equation}
\label{5.9}
Z_{n+1}=\left[I-\frac{h}{2}J\right]^{-1}\left[I+\frac{h}{2}J\right]Z_{n}+\left[I-\frac{h}{2}J\right]^{-1}\left[I+\frac{h}{2}J\right]\begin{pmatrix}
\sigma \\ 0
\end{pmatrix}\Delta W_{n}
\end{equation}
will be obtained.

\begin{figure*}[t]
\centering\subfigure{
\includegraphics[width=6cm,height=5cm]{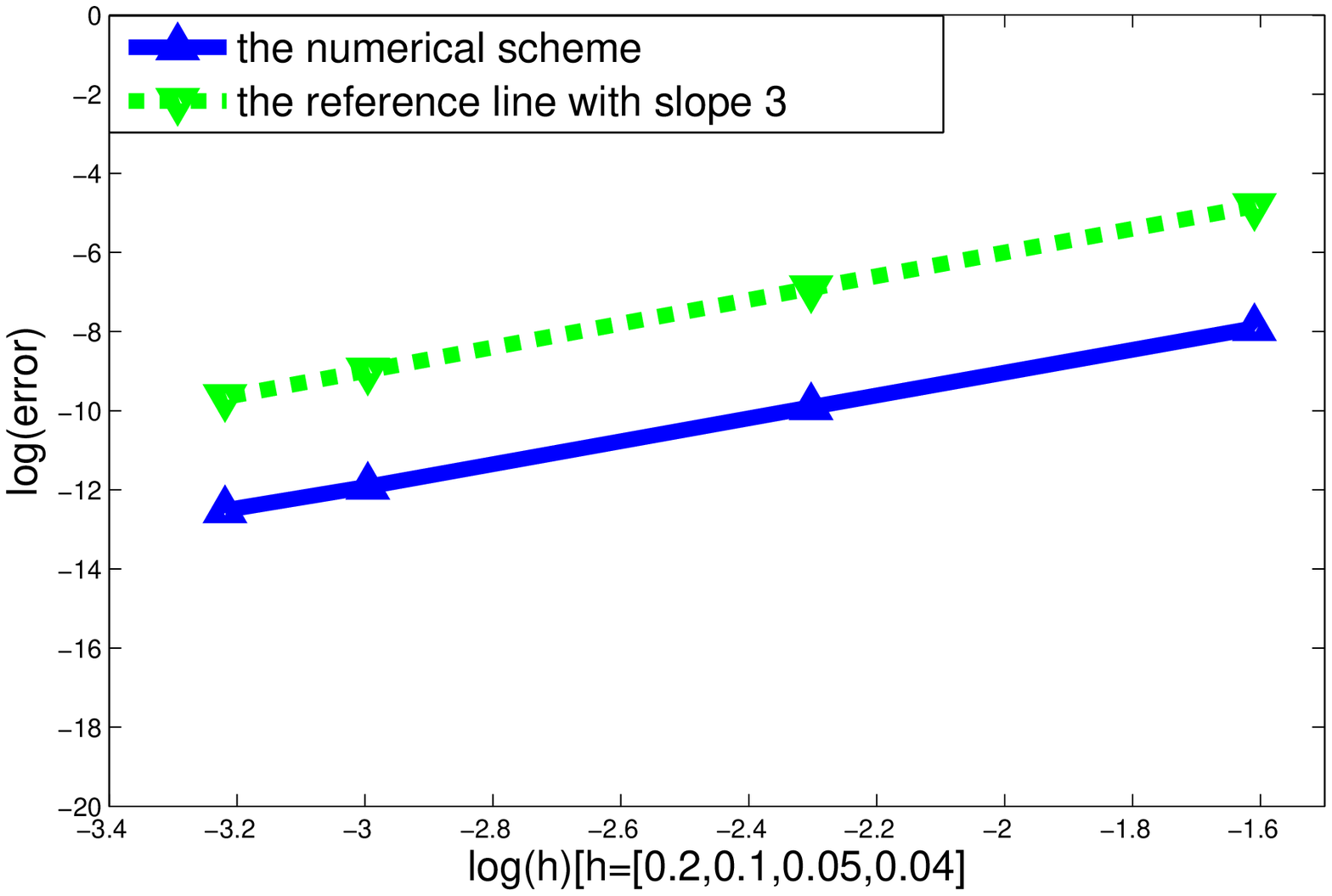}}
\subfigure{
\includegraphics[width=6cm,height=5cm]{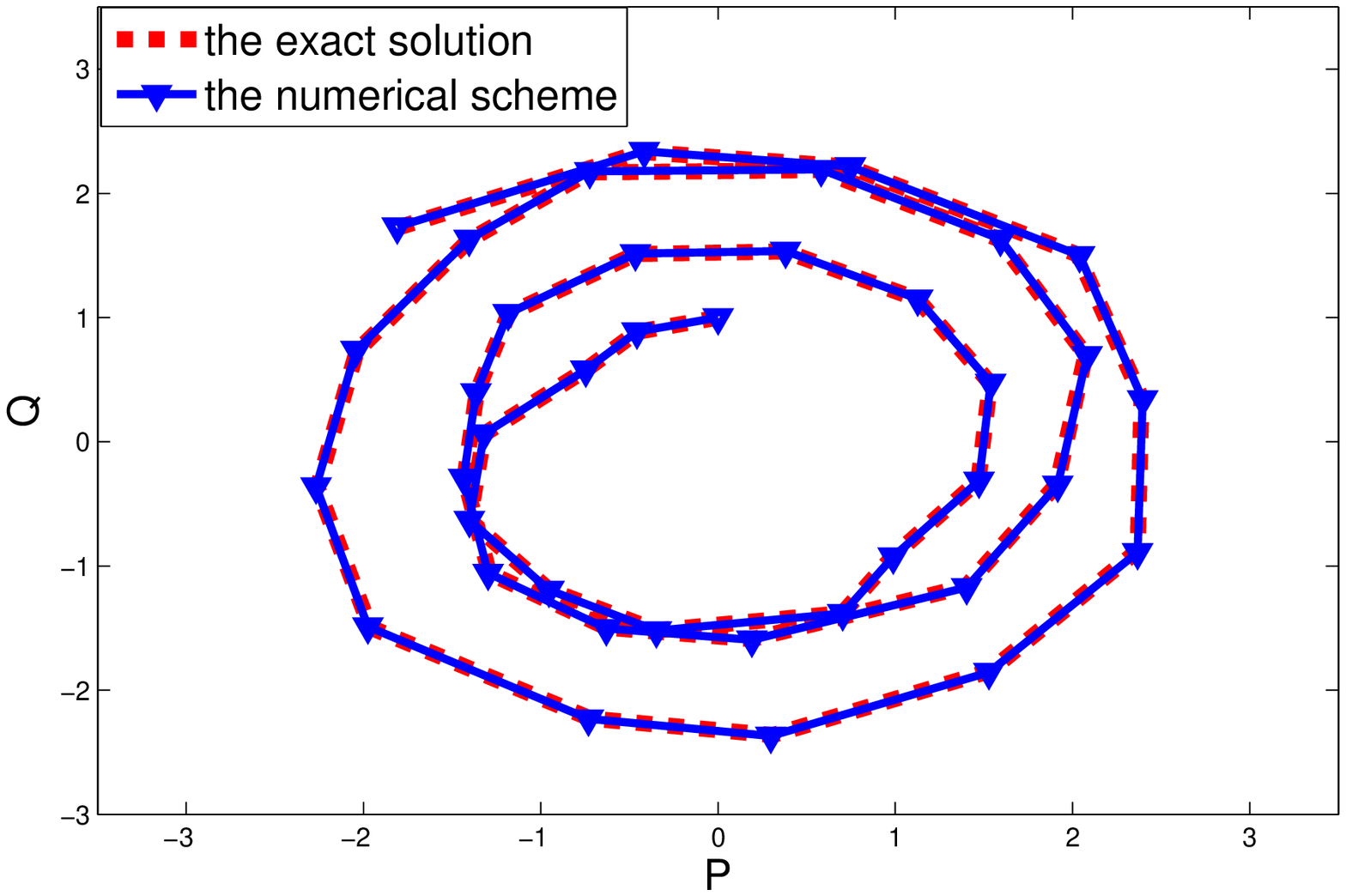}}
\caption{The mean-square convergence order of the scheme (\ref{5.8}) (left), and a sample path trajectory of the numerical solution (\ref{5.8}) (right).}\label{p1}
\end{figure*}
\begin{figure*}[t]
\centering\subfigure{
\includegraphics[width=6cm,height=5cm]{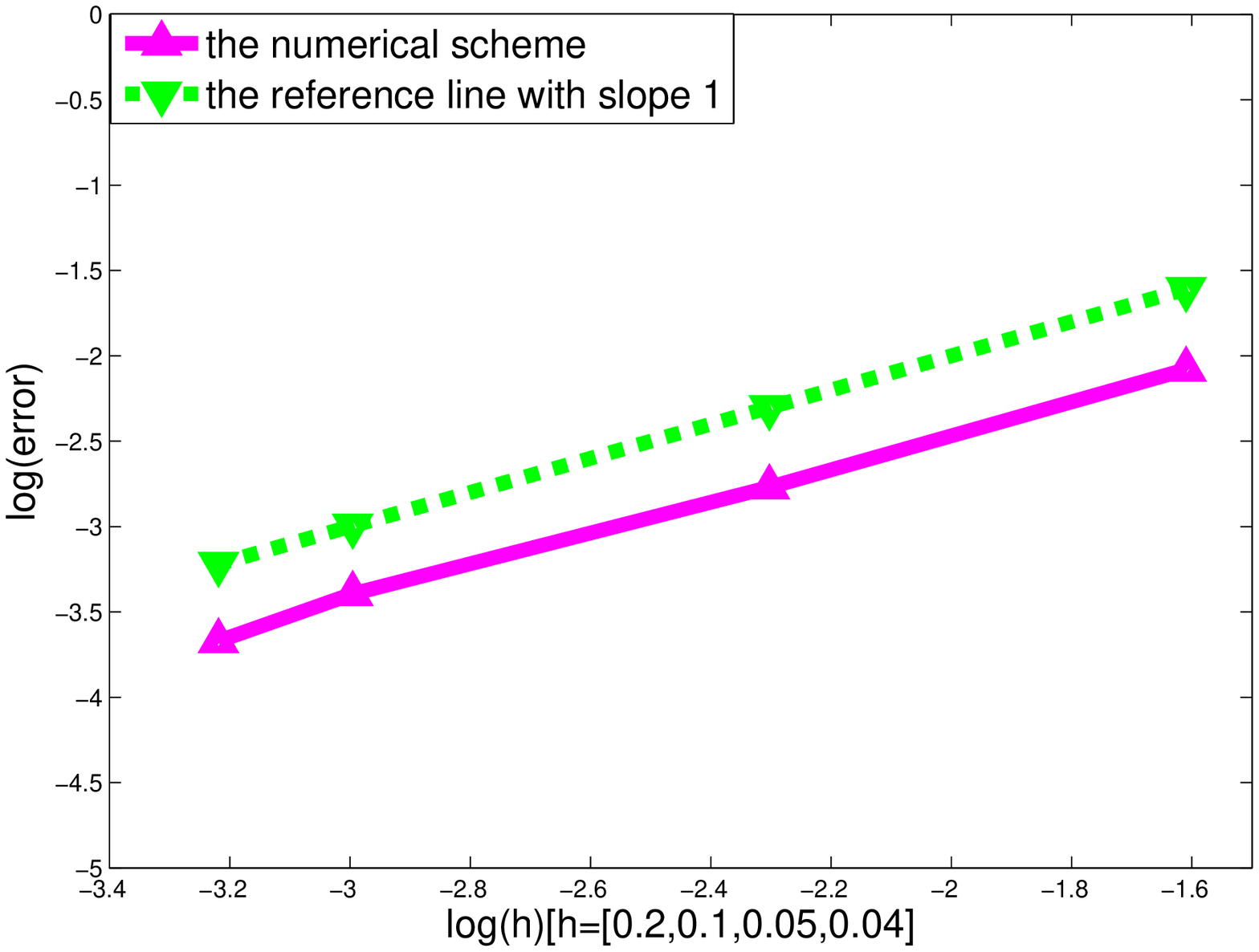}}
\subfigure{
\includegraphics[width=6cm,height=5cm]{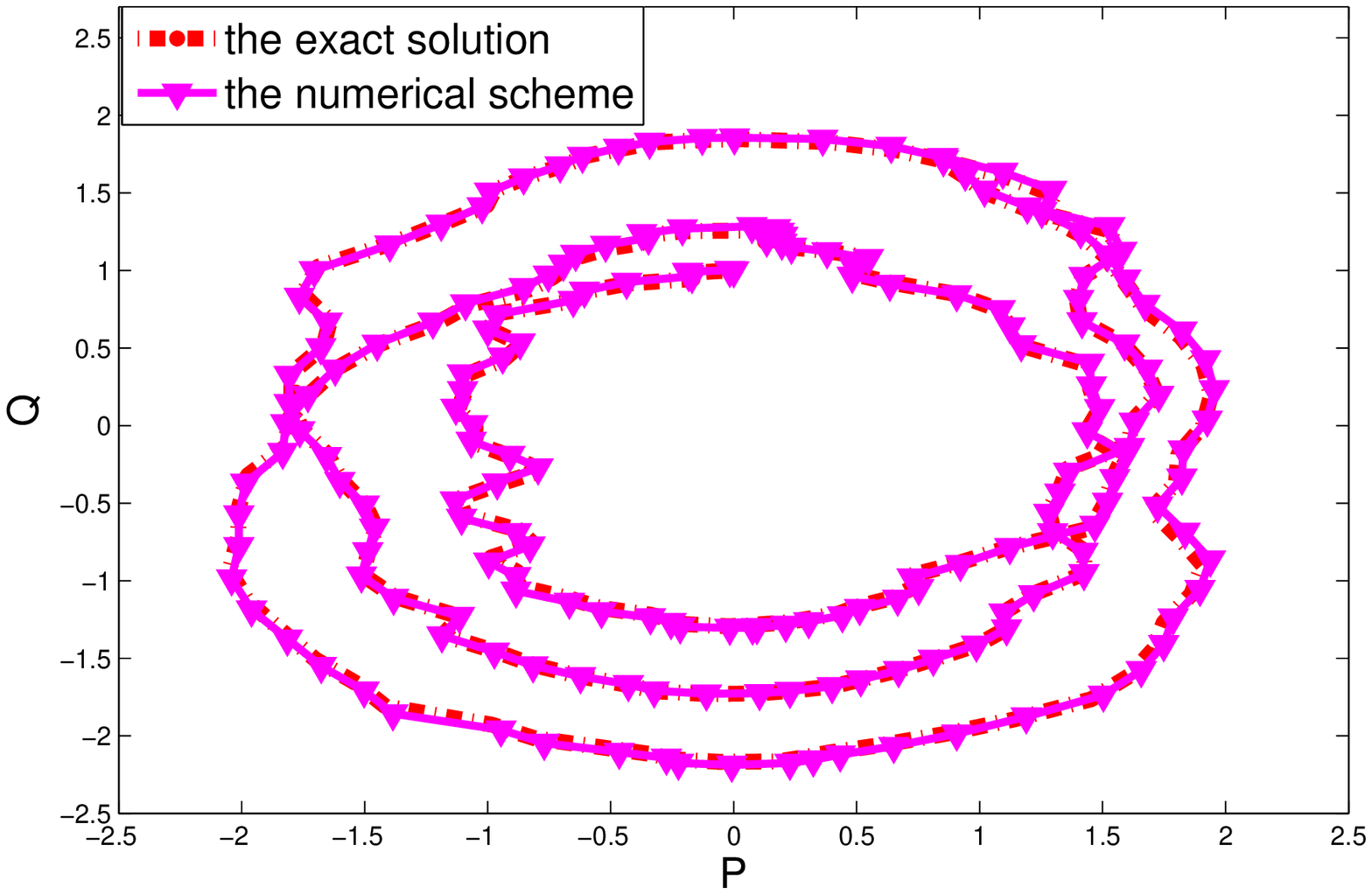}}
\caption{The mean-square convergence order of the scheme (\ref{5.9}) (left), and a sample path trajectory of the numerical solution (\ref{5.9}) (right).}\label{p3}
\end{figure*}

With respect to the above numerical methods (\ref{5.8}) and (\ref{5.9}), our numerical tests focus on two aspects of view mainly. One is to sketch the numerical mean, i.e. the sample average of $p_{n}^2+q_{n}^2$ directly, and then to compare their accordance with the reference line, the slop of which is the rate of the linear growth of the second moment of the solution (\ref{5.6}). The other is to show that the proposed numerical schemes (\ref{5.8}) and (\ref{5.9}) preserve the oscillation property.

For the numerical scheme (\ref{5.8}), the left panel of Figure \ref{p1} plots the value $\ln|{\bf{E}}(P(T)+Q(T))^2$$-{\bf{E}}(P_{N}+Q_{N})^2|$
(blue-dotted line) against
$\ln h$ for four different step sizes $h=[0.2;0.1;0.05;0.04]$ at $T =20$, where $T$ is
the subindex of one of the discrete time point such that $t_{N} =T$, and the $(P(T);Q(T))$ and
$(P_{N};Q_{N})$ represent the phase point of the exact solution and the numerical
method at time $T$, respectively. It can be seen that
the mean-square order of the scheme (\ref{5.8}) is 3, as indicated by the reference lines of slope 3. The expectation $\bf{E}$ is approximated by taking average over 500 sample paths. The right panel draws one sample trajectory of the numerical method
(\ref{5.8}), and that of the stochastic differential equation (\ref{5.7}), where near coincidence can be seen.

Similar to the test for the numerical scheme (\ref{5.8}), we observe the mean-square convergence order of the numerical method (\ref{5.9}) is 1, as illustrated by the left panel of Figure \ref{p3}. The data setting for the
left panel of Figure \ref{p3} is $p_0=0$, $q_0=1$, $T=20$ and
$h=[0.2;0.1;0.05;0.04]$, and the expectation $\bf{E}$ is approximated by taking average over 1000 sample paths. The right panel draws a sample phase trajectory of the exact solution (red-doted line), the numerical method (\ref{5.9}) (pink-solid line). The coincidence of a sample phase trajectory of the numerical method with that of the exact solution is also visible. The data here are the same with that for the right panel of Figure \ref{p1}.

\begin{figure*}[t]
\centering\subfigure{
\includegraphics[width=6cm,height=5cm]{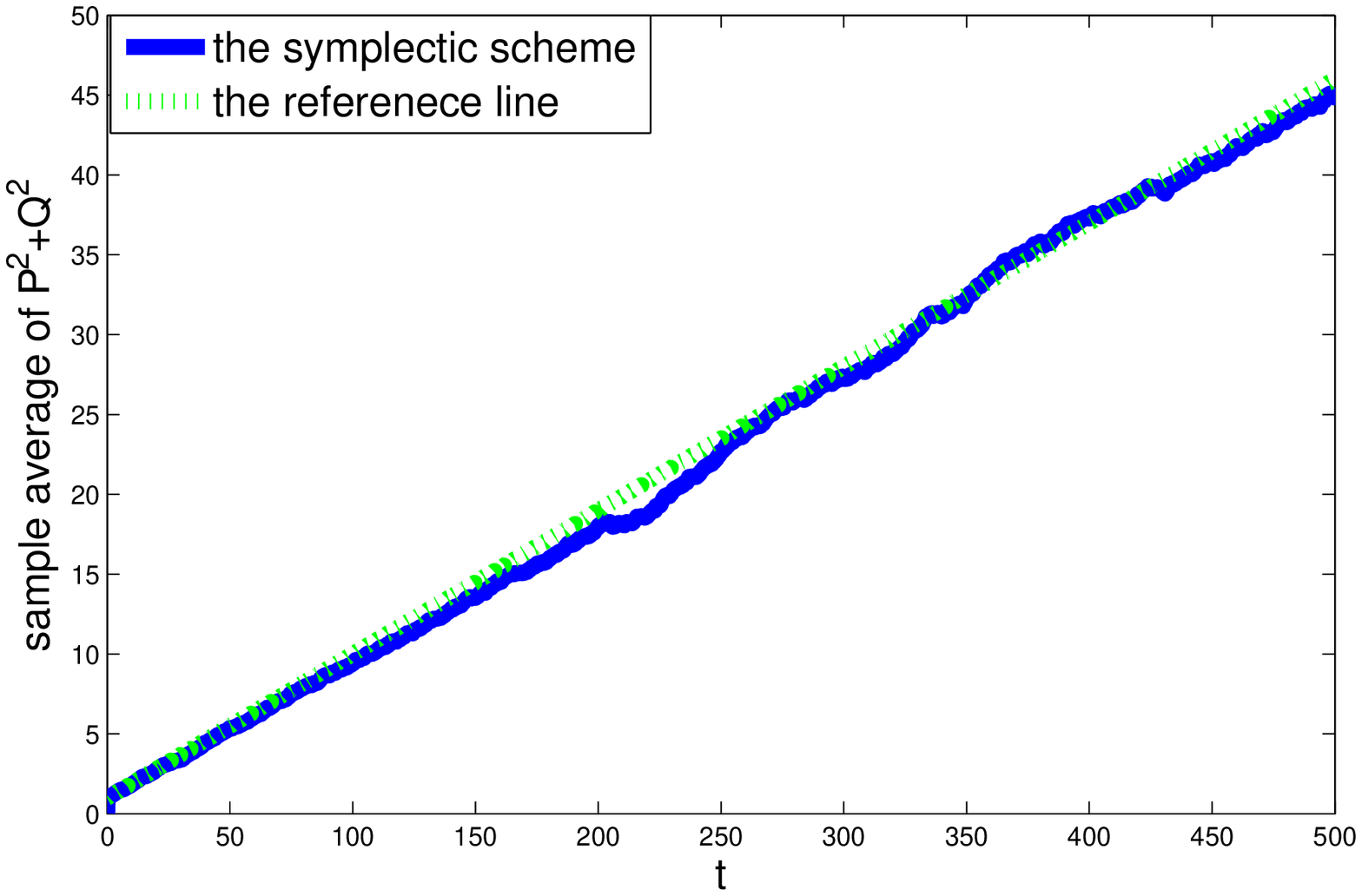}}
\subfigure{
\includegraphics[width=6cm,height=5cm]{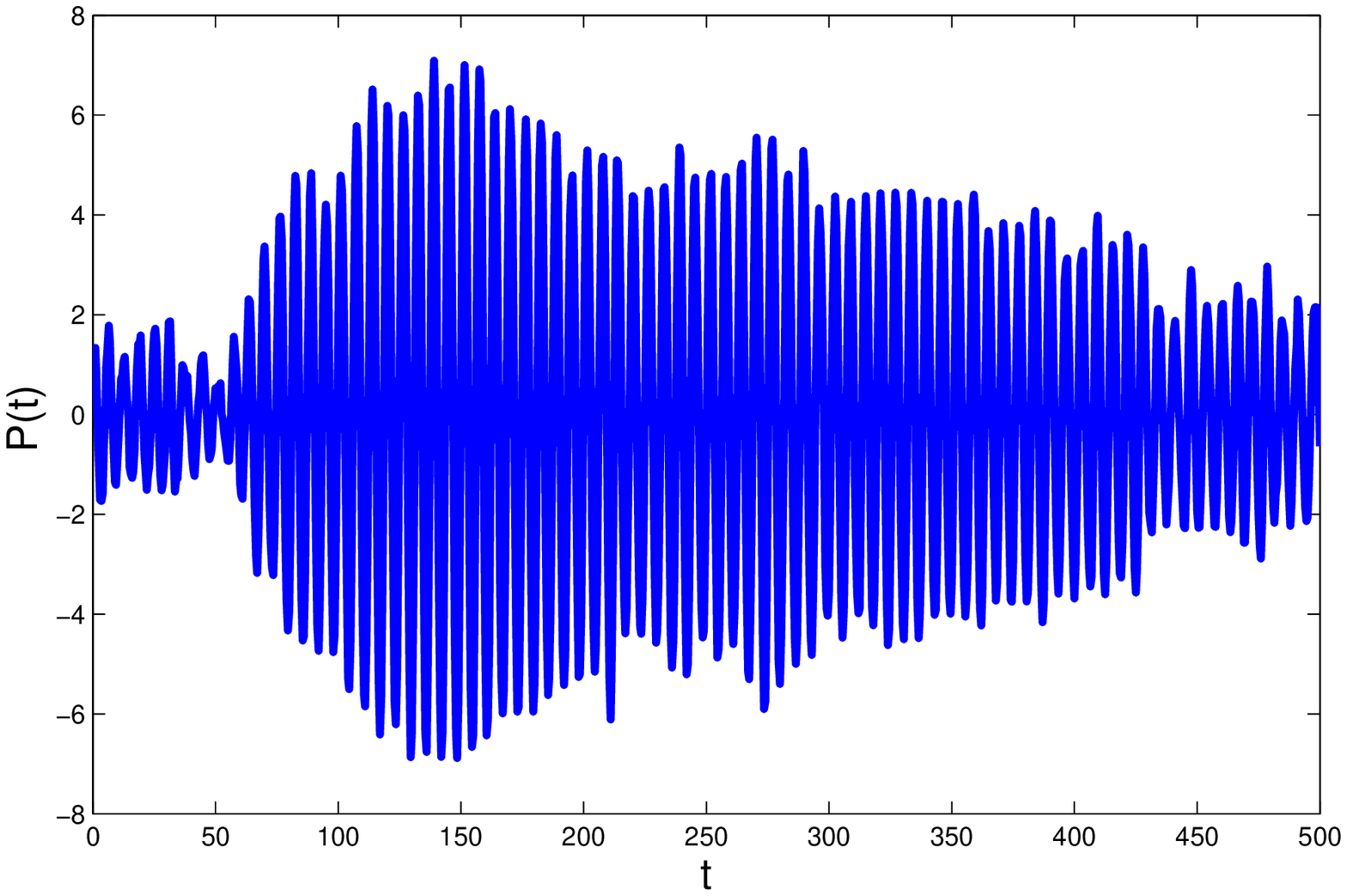}}
\caption{Preservation of the linear growth property of the numerical scheme (\ref{5.8}) (left), Oscillation of the numerical solution  (\ref{5.8}) (right)..}\label{p2}
\end{figure*}
\begin{figure*}[t]
\centering\subfigure{
\includegraphics[width=6cm,height=5cm]{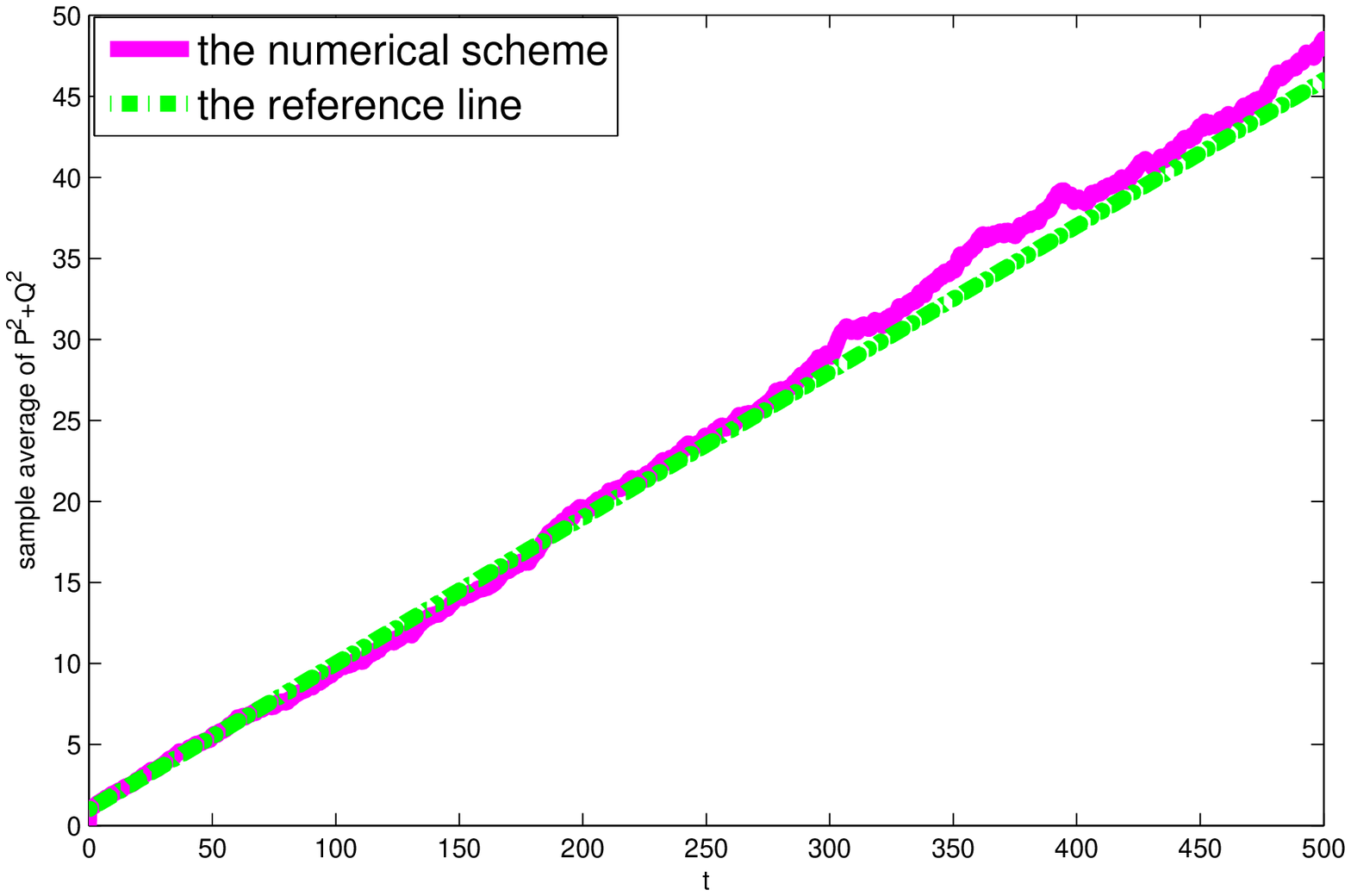}}
\subfigure{
\includegraphics[width=6cm,height=5cm]{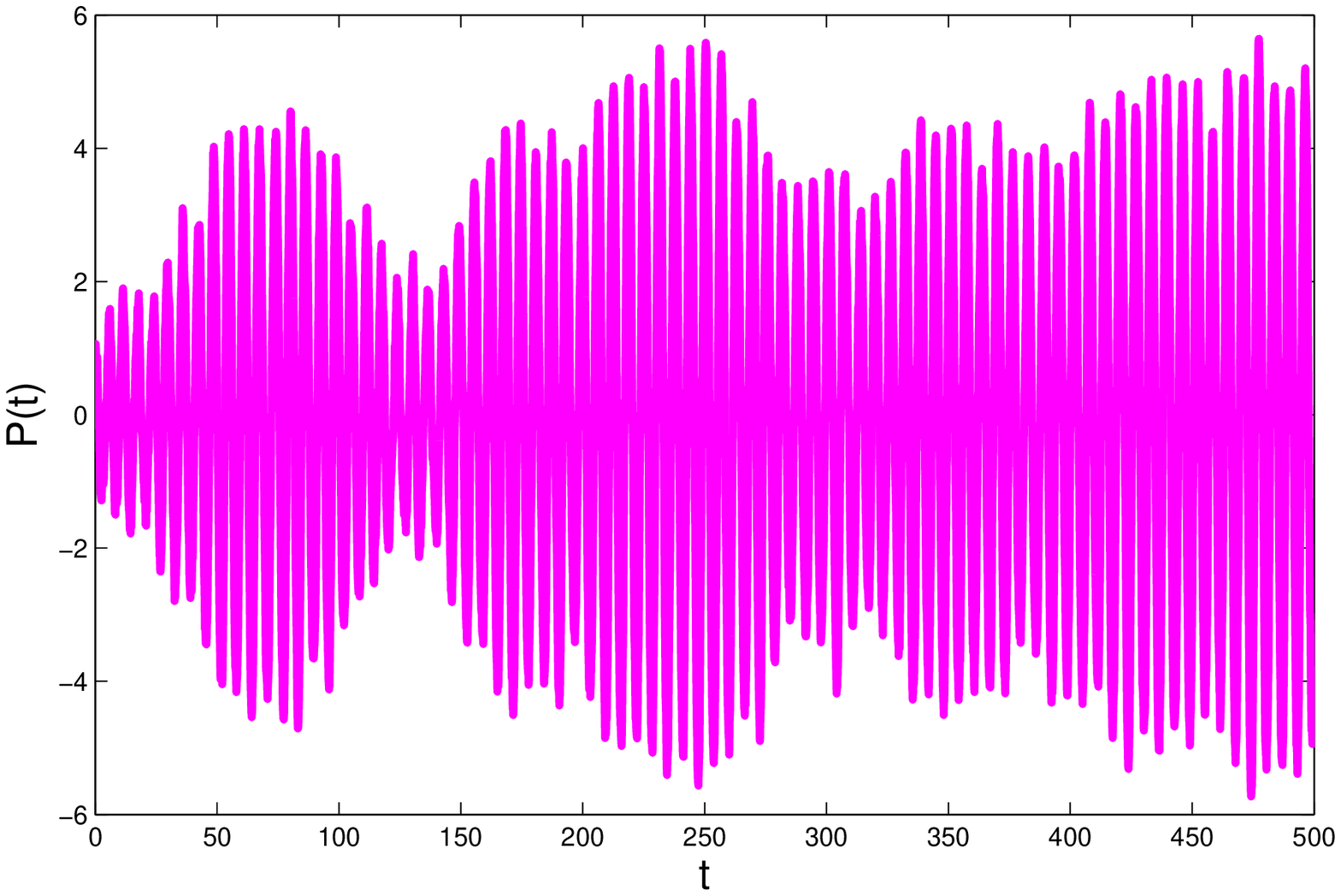}}
\caption{Preservation of the linear growth property of the numerical scheme (\ref{5.9}) (left), Oscillation of the numerical solution  (\ref{5.9}) (right).}\label{p4}
\end{figure*}

In the left panels of Figure \ref{p2} and \ref{p4}, the quantity ${\bf{E}}(p_{n}^2+q_{n}^2)$ with respect to the numerical solution $p_{n}$, $q_{n}$, $n=0,1,2,\cdots,5000$ produced by the numerical schemes (\ref{5.8}) (Figure \ref{p2}) and (\ref{5.9}) (Figure \ref{p4}) is simulated through taking sample average over 500 numerical sample paths created by corresponding numerical schemes applied to the linear stochastic system \ref{5.6}. Here $\sigma=0.3$, $t\in [0,500]$ and the stepsize is 0.1. The reference straight line (green-dashed) has slope 0.09, which is equal to $\sigma ^2$. It can be seen from the two figures that both the numerical schemes (\ref{5.8}) and (\ref{5.9}) preserve the linear growth property of the second moment of the solution for the original system (\ref{5.6}). Moreover, the long time oscillation behavior of
the numerical solutions (\ref{5.8}) and (\ref{5.9}) are illustrated in the right panels of Figure \ref{p2} and \ref{p4}.
The property that the numerical solutions (\ref{5.8}) and (\ref{5.9}) have infinite zeros is obvious via the Figures.

\section{Conclusion}
\label{6}
We construct stochastic symplectic numerical methods
using Pad$\acute e$ approximation, for linear stochastic Hamiltonian systems, and special
stochastic Hamiltonian systems with additive noises. Applications of the method to two examples, i.e Kubo oscillator and a linear stochastic oscillator, succeed in constructing
symplectic numerical solutions based on Pad$\acute e$ approximation which inherit the properties of the original systems. Numerical experiments
show the mean-square convergence orders of the proposed schemes. For the deterministic situations, it is known that the numerical schemes based on the Pad$\acute e$ approximation are $A$-stable (\cite{Aboanber}) under appropriate conditions. However, stochastic stability of our methods still need further investigation.


\begin{thebibliography}{00}
\bibitem{feng}K. Feng, M.ZH. Qin. Symelectic Geometric Algorithms for Hamiltonian Systems, Springer-Verlag Berlin Heidelberg (2010).
\bibitem{hairer}E. Hairer, C. Lubich, and G. Wanner, Geometric Numerical Integration, Springer-Verlag Berlin Heidelberg (2002).
\bibitem{kloeden}P.E. Kloeden and E. Platen. Numerical solution of stochastic differential equations[M], Berlin Heidelberg: Springer-Verlag, 1992.
\bibitem{mao}X.R. Mao. Stochastic differential equations and applications[M], 2nd Edition, Horwood Publishing Limited, 2007.
\bibitem{milbook}G.N. Milstein, M.V. Tretyakov. Stochastic numerics for mathematical physics[M], Berlin Heidelberg: Springer-Verlag, 2004.
\bibitem{mil1}G.N. Milstein, Y.M. Repin and M.V. Tretyakov. Symplectic integration of Hamiltonian systems with additive noise[C], SIAM J. Numer. Anal., 2002, 39: 2066-2088 .
\bibitem{mil2}G.N. Milstein, Y.M. Repin and M.V. Tretyakov. Numerical methods for stochastic systems preserving symplectic structure[C], SIAM J. Numer. Anal.,2002,40: 1583-1604.
\bibitem{oksen}B. ${\O}$ksendal. Stochastic Differential Equations[M], 6th edition, Springer, 2003.
\bibitem{Ikeda}N. Ikeda and S. Watanabe, Stochastic Differential Equations and Diffusion Processes. North-Holland Publ. Comp., Amsterdam, 1981.
\bibitem{Aboanber}A. E. Aboanber, On Pade' Approximations to the Exponential Function and Application to the Point Kinetics Equations. Great Britain, Elsevier Ltd. pp. 347-~68, 2004.
\bibitem{melbo}A.H. Str${\o}$mmen Melb${\o}$, D.J. Higham, Numerical simulation of a linear stochastic oscillator with additive noise, Appl. Numer. Math., 51 (1), 89-99 (2004).
\bibitem{sanz}J.M. Sanz-Serna, M.P. Calvo, Numerical Hamiltonian problems, volume 7 of Applied Mathematics and Mathematical Computation, Chapman \& Hall, London (1994).
\bibitem{DCO}D. Cohen, On the numerical discretisation of stochastic oscillators, Mathematics and Computers in Simulation, 82(2012),1478-1495.
\bibitem{wang}L.J. Wang, Variational integrators and generating functions for stochastic Hamiltonian
systems, Ph.D thesis, Karlsruhe Institute of Technology, KIT Scientific Publishing (2007).
\bibitem{Markus}L. Markus, A. Weerasinghe, Stochastic oscillators, J. Differential Equations, 71 (2), 288-314 (1988).
\end{thebibliography}
\end{document}